\documentclass{article}

    \usepackage[preprint]{neurips_2021}

\usepackage[utf8]{inputenc} 
\usepackage[T1]{fontenc}    
\usepackage{hyperref}       
\usepackage[normalem]{ulem} 
\usepackage{url}            
\usepackage{booktabs}       
\usepackage{amsfonts}       
\usepackage{nicefrac}       
\usepackage{microtype}      
\usepackage{xcolor}         

\usepackage{microtype}
\usepackage{graphicx}
\usepackage{subfigure}
\usepackage{booktabs} 
\usepackage{mathtools}
    \allowdisplaybreaks
    
\usepackage{amssymb}
    \DeclareMathOperator*{\rank}{\mathrm{rank}}
    \DeclareMathOperator*{\tr}{\mathrm{tr}}
\usepackage{amsthm}
\theoremstyle{plain} 
    \newtheorem{theorem}{Theorem}
    \newtheorem{corollary}{Corollary}
    \newtheorem{lemma}{Lemma}
    
\theoremstyle{definition}
    \newtheorem{assumption}{Assumption}
    \newtheorem{definition}{Definition}
    \newtheorem{example}{Example}
\theoremstyle{remark}
    \newtheorem{remark}{Remark}

\usepackage{makecell}

\usepackage{ulem}
\usepackage{booktabs}
\usepackage{multirow}
\usepackage{comment}
\usepackage{verbatimbox}
\usepackage{caption}
\usepackage{color}
\usepackage{verbatimbox}

\usepackage{algorithm}
\usepackage{algorithmicx}
\usepackage{algpseudocode}
\algnewcommand{\algorithmicand}{\textbf{and }}
\algnewcommand{\algorithmicor}{\textbf{or }}
\algnewcommand{\OR}{\algorithmicor}
\algnewcommand{\AND}{\algorithmicand}

\title{General Low-rank Matrix Optimization:\\ Geometric Analysis and Sharper Bounds}

\author{%
  Haixiang Zhang\\
  Department of Mathematics\\
  University of California, Berkeley\\
  Berkeley, CA 94704 \\
  \texttt{haixiang\_zhang@berkeley.edu} 
  \And
  Yingjie Bi\\
  Department of IEOR\\
  University of California, Berkeley\\
  Berkeley, CA 94704 \\
  \texttt{yingjiebi@berkeley.edu} \\
  \And
  Javad Lavaei\\
  Department of IEOR\\
  University of California, Berkeley\\
  Berkeley, CA 94704 \\
  \texttt{lavaei@berkeley.edu}
}

\begin{document}

\maketitle

\vspace{-1.5em}
\begin{abstract}
  This paper considers the global geometry of general low-rank minimization problems via the Burer-Monterio factorization approach. For the rank-$1$ case, we prove that there is no spurious second-order critical point for both symmetric and asymmetric problems if the rank-$2$ RIP constant $\delta$ is less than $1/2$. Combining with a counterexample with $\delta=1/2$, we show that the derived bound is the sharpest possible. For the arbitrary rank-$r$ case, the same property is established when the rank-$2r$ RIP constant $\delta$ is at most $1/3$. We design a counterexample to show that the non-existence of spurious second-order critical points may not hold if $\delta$ is at least $1/2$. In addition, for any problem with $\delta$ between $1/3$ and $1/2$, we prove that all second-order critical points have a positive correlation to the ground truth. Finally, the strict saddle property, which can lead to the polynomial-time global convergence of various algorithms, is established for both the symmetric and asymmetric problems when the rank-$2r$ RIP constant $\delta$ is less than $1/3$. The results of this paper significantly extend several existing bounds in the literature. 
\end{abstract}
\vspace{-1.5em}

\section{Introduction}
\label{eqn:intro}

Given the natural numbers $n$, $m$ and $r$, consider the low-rank matrix optimization problems
\begin{align}\label{eqn:sym-original} 
    \min_{M\in\mathbb{R}^{n\times n}}~f_s(M) \quad \mathrm{s.t.}~ \rank(M) \leq r,\quad M^T = M,\quad M \succeq0
\end{align}
and 
\begin{align}\label{eqn:asym-original} 
    \min_{M\in\mathbb{R}^{n\times m}}~ f_a(M)  \quad \mathrm{s.t.}~ \rank(M) \leq r,
\end{align}
where the functions $f_s(\cdot)$ and $f_a(\cdot)$ are twice continuously differentiable. Problems \eqref{eqn:sym-original}-\eqref{eqn:asym-original} are referred to as the \emph{symmetric} and the \emph{asymmetric} problem, respectively. In addition, we call these problems \emph{linear} if the objective function is induced by a linear measurement operator, i.e.,
\begin{align*} 
f(M) = \textstyle{\frac12} \|\mathcal{A}(M) - b\|_F^2 
\end{align*}
for some vector $b\in\mathbb{R}^p$ and linear operator $\mathcal{A}$ mapping each matrix $M$ to a vector in $\mathbb{R}^p$, where $f(M)$ denotes either $f_s(M)$ or $f_a(M)$. Those problems not fitting into the above model are called \emph{nonlinear}. Low-rank optimization problems arise in a wide range of applications, e.g., matrix completion \citep{candes2009exact,recht2010guaranteed}, phase synchronization \citep{singer2011angular,boumal2016nonconvex}, phase retrieval \citep{shechtman2015phase}; see \cite{chen2018harnessing,chi2019nonconvex} for an overview of the topic. To overcome the non-convex rank constraint, one may resort to convex relaxations. The approach of replacing the rank constraint with a nuclear norm regularization is proven to provide the optimal sample complexity \citep{candes2009exact,recht2010guaranteed,candes2010power}. However, solving the convexified problems involves computing a Singular Value Decomposition (SVD) in each iteration and results in heavy computational burdens. Along with the issue of large space complexities, the convexification approach is impractical for large-scale problems. Therefore, it is important to design efficient alternative methods with similar theoretical guarantees.

\subsection{Burer--Monterio factorization and basic properties}

Instead of directly solving convex relaxations of problems \eqref{eqn:sym-original}-\eqref{eqn:asym-original}, we consider a computationally efficient approach, namely the Burer--Monterio factorization \citep{burer2003nonlinear}. The factorization approach is based on the observation that any matrix $M\in\mathbb{R}^{n\times m}$ with rank at most $r$ can be written in the form of $UV^T$, where $U\in\mathbb{R}^{n\times r}$ and $V\in\mathbb{R}^{m\times r}$. Then, the asymmetric problem \eqref{eqn:asym-original} is equivalent to
\begin{align}\label{eqn:asym} 
    \min_{U\in\mathbb{R}^{n\times r},V\in\mathbb{R}^{m\times r}}~ h_a(U,V),
\end{align}
where $h_a(U,V) := f_a(UV^T)$. Similarly, the symmetric problem \eqref{eqn:sym-original} is equivalent to
\begin{align}\label{eqn:sym} 
    \min_{U\in\mathbb{R}^{n\times r}}~h_s(U),
\end{align}
where $h_s(U) := f_s(UU^T)$. The Burer-Monterio factorization provides a natural parameterization of the low-rank structure of the unknown solution, and reformulates problems \eqref{eqn:sym-original}-\eqref{eqn:asym-original} as unconstrained optimization problems. In addition, the number of variables reduces from $O(n^2)$ or $O(nm)$ to as low as $O(n)$ or $O(n+m)$ when $r\ll \min\{n,m\}$. However, the reformulated problems are highly non-convex, and $\mathcal{NP}$-hard to solve in general. On the other hand, these problems share a specific non-convex structure, which makes it possible to utilize the structure and design efficient algorithms to find a global optimum under some conditions. In addition to the special structure, a regularity condition, named the Restricted Isometry Property, is required to guarantee the convergence of common iterative algorithms. We state the following two definitions only in the context of the symmetric problem since the corresponding definitions for the asymmetric problem are similar.
\begin{definition}[\cite{recht2010guaranteed,zhu2018global}]
Given natural numbers $r$ and $t$, the function $f_s(\cdot)$ is said to satisfy the \textbf{Restricted Isometry Property} (RIP) of rank $(2r,2t)$ for a constant $\delta\in[0,1)$,  denoted as $\delta$-RIP$_{2r,2t}$, if 
\[ (1-\delta)\| K\|_F^2 \leq \left[\nabla^2 f_s(M)\right] (K,K) \leq (1+\delta)\| K\|_F^2 \]
holds for all matrices $M,K\in\mathbb{R}^{n\times n}$ such that $\rank(M)\leq 2r,\rank(K) \leq 2t$, where $\left[\nabla^2 f_s(M)\right](\cdot,\cdot)$ is the curvature of the Hessian at point $M$.
\end{definition}
The RIP property is satisfied by the objective functions in a wide range of applications, either on the entire low-rank manifold or on a part of the manifold that is of interest; examples include matrix sensing \citep{candes2009exact,recht2010guaranteed}, matrix completion \citep{candes2010power}, and matrix sensing with non-Gaussian noise \citep{davenport20141}. For instance, in the case of linear measurements with a Gaussian model, \cite{candes2009exact} showed that $O\left(nr/\delta^2\right)$ samples are enough to ensure the $\delta$-RIP$_{2r,2r}$ property with high probability. Since the analysis of the statistical behavior of problems \eqref{eqn:asym}-\eqref{eqn:sym} is not in the scope of this work, we focus on the case when the RIP property is satisfied by the objective function. We note that the RIP property is equivalent to the restricted strongly convex and smooth property defined in \cite{wang2017unified,park2018finding,zhu2021global} with the condition number $(1+\delta)/(1-\delta)$. Intuitively, the RIP property implies that the Hessian matrix is close to the identity tensor when the perturbation is restricted to be low-rank. This intuition naturally leads to the following definition.
\begin{definition}[\cite{biglobal}]
Given a natural number $r$, the function $f_s(\cdot)$ is said to satisfy the \textbf{Bounded Difference Property} (BDP) of rank $2r$ for a constant $\kappa\geq0$, denoted as $\kappa$-BDP$_{2r}$, if 
\[ \left|\left[\nabla^2 f_s(M) - \nabla^2 f_s(M')\right] (K,L) \right| \leq \kappa\| K\|_F\| L\|_F \]
holds for all matrices $M,M',K,L\in\mathbb{R}^{n\times n}$ such that $\rank(M),\rank(M'),\rank(K),\rank(L)\leq 2r$. 
\end{definition}
It has been proven in \cite[Theorem 1]{biglobal} that those functions satisfying the $\delta$-RIP$_{2r,2r}$ property also satisfy the $4\delta$-BDP$_{2r}$ property. 
With the RIP property, there are basically two categories of algorithms that can solve the factorized problem in polynomial time. Algorithms in the first category require a careful initialization so that the initial point is already in a small neighbourhood of a global optimum, and a certain local regularity condition in the neighbourhood ensures that local search algorithms will converge linearly to a global optimum; see \cite{tu2016low,bhojanapalli2016dropping,park2018finding} for a detailed discussion. The other class of algorithms is able to converge globally from a random initialization. The convergence of these algorithms is usually established via the geometric analysis of the landscape of the objective function. One of the important geometric properties is the strict saddle property \citep{sun2018geometric}, which combined with the smoothness properties can guarantee the global polynomial-time convergence for various saddle-escaping algorithms \citep{jin2017escape,jin2018accelerated,sun2018geometric,huang2019perturbed}. For the linear case, \cite{ge2016matrix,ge2017no} proved the strict saddle property for both problems \eqref{eqn:asym}-\eqref{eqn:sym} when the RIP constant is sufficiently small. More recently, \cite{zhu2021global} extended the results to the nonlinear asymmetric case. Moreover, a weaker geometric property, namely the non-existence of \emph{spurious (non-global) second-order critical points}, has been established for both problems when the RIP constant is small \citep{li2019non,ha2020equivalence}. We note that second-order critical points are points that satisfy the first-order and the second-order critical necessary conditions, and thus the result of second-order critical points implies the non-existence of spurious local minima. Under certain regularity conditions, this weaker property is also able to guarantee the global convergence from a random initialization without an explicit convergence rate \citep{lee2016gradient,panageas2016gradient}. Please refer to Table \ref{tab:result} for a summary of the state-of-the-art results.
%

Most of the afore-mentioned papers are based on the following assumption on the low-rank critical points of the functions $f_s(\cdot)$ and $f_a(\cdot)$.
\begin{assumption}\label{asp:1}
The function $f_a(\cdot)$ has a first-order critical point $M^*_a$ such that $\rank(M^*_a) \leq r$. Similarly, the function $f_s(\cdot)$ has a first-order critical point $M_s^*$ that is symmetric, positive semi-definite and of rank at most $r$. 
\end{assumption}
This assumption is inspired by the noiseless matrix sensing problem in the linear case for which the non-negative objective function becomes zero (the lowest value possible) at the true solution. This is a natural property of the matrix sensing problem for nonlinear measurement models as well. Under the above assumption and the RIP property, \cite{zhu2018global} proved that $M^*_s$ and $M_a^*$ are the unique global minima of problems \eqref{eqn:sym-original}-\eqref{eqn:asym-original}. 
\begin{theorem}[\cite{zhu2018global}]
If the functions $f_s(\cdot)$ and $f_a(\cdot)$ satisfy the $\delta$-RIP$_{2r,2r}$ property, then the critical points $M^*_s$ and $M_a^*$ are the \emph{unique global minima} of problems \eqref{eqn:sym-original}-\eqref{eqn:asym-original}.
\end{theorem}
Given a solution $(U^*,V^*)$ to problem \eqref{eqn:asym}, we observe that $(U^*P,V^*P^{-T})$ is also a solution for any invertible $P\in\mathbb{R}^{r\times r}$. This redundancy may induce an extreme non-convexity on the landscape of the objective function. To reduce this redundancy, \cite{tu2016low} considered the regularized problem
\begin{align}
    \label{eqn:asym-reg} \min_{U\in\mathbb{R}^{n\times r},V\in\mathbb{R}^{m\times r}}~ \rho(U,V),
\end{align}
where
\[ \rho(U,V) := h_a(U,V) + \frac{\mu}{4} \cdot g(U,V) \]
with a constant $\mu>0$ and the regularization term  
\[ g(U,V) := \|U^TU - V^TV\|_F^2. \]
The regularization term is introduced to ``balance'' the magnitudes of $U^*$ and $V^*$. \cite{zhu2018global} showed that the regularization term does not introduce bias and thus problem \eqref{eqn:asym-reg} is equivalent to the original problem \eqref{eqn:asym-original}.
\begin{theorem}[\cite{zhu2018global}]
Any first-order critical point $(U^*,V^*)$ of problem \eqref{eqn:asym-reg} satisfies $(U^*)^TU^* = (V^*)^TV^*$. Moreover, problems \eqref{eqn:asym-original} and \eqref{eqn:asym-reg} are equivalent.
\end{theorem}
Detailed optimality conditions for problems \eqref{eqn:sym-original}-\eqref{eqn:asym-reg} are provided in the appendix.

\subsection{Contributions}
\label{sec:contribution}

In this work, we analyze the geometric properties of problems \eqref{eqn:sym}-\eqref{eqn:asym-reg}. Novel analysis methods are developed to obtain less conservative conditions for guaranteeing benign landscapes for both problems. We note that, unlike the linear measurements case, the RIP constant of nonlinear problems may not concentrate to $0$ as the number of samples increases. Therefore, a sharper RIP bound leads to theoretical guarantees that hold under less stringent statistical requirements. In addition, even if the RIP constant concentrates to $0$ when more samples are included, there may only be a limited number of samples available, either due to the constraints of specific applications or to the great expense of taking more samples. Hence, obtaining a sharper RIP bound is essential for many applications. We summarize our results in Table \ref{tab:result}. More concretely, the contributions of this paper are three-folds. 

\begin{table}[t]
\small
\caption{Comparison of the state-of-the-art results and our results. Here $\delta_{2r,2t},\kappa$ are the RIP$_{2r,2t}$ and BDP$_{2r}$ constants of $f_s(\cdot)$ or $f_a(\cdot)$, respectively. Constant $\alpha(M^*_a)\in(0,1)$ only depends on $M^*_a$.}\label{tab:result}
  \begin{center}  
    \begin{tabular}{cccccc}
          \toprule[2pt]
          \multicolumn{2}{c}{\textbf{}} & \multicolumn{2}{c}{\textbf{No Spurious Second-order Critical Pts.}}    & \multicolumn{2}{c}{\textbf{Strict Saddle Property}}  \\
          \cline{3-6}
          \multicolumn{2}{c}{\textbf{Problem Setups}} & Existing & Ours    & Existing & Ours   \\
          \midrule[1pt]
          \multirow{2}{*}{\makecell*[c]{\textbf{Rank-$1$} \\ \textbf{Sym.} }} & \textbf{Linear}    & \makecell*[c]{$\delta_{2,2}<\frac12$\\ \citep{zhang2019sharp}}  & $\delta_{2,2}<\frac12$    & -  & -           \\
          & \textbf{Nonlinear}    & \makecell*[c]{$\delta_{2,2} < \frac{2-O(\kappa)}{4+O(\kappa)}$\\\citep{biglobal}}  & $\delta_{2,2}<\frac12$     & -  & -           \\
          \midrule[1pt]
          \makecell*[c]{\textbf{Rank-$1$} \\ \textbf{Asym.} } & \makecell*[c]{\textbf{Linear} \\ \textbf{\& Nonlinear} }    & -  & $\delta_{2,2}<\frac12$    & -  & -           \\
          \midrule[1pt]
          \multirow{2}{*}{\makecell*[c]{\textbf{Rank-$r$} \\ \textbf{Sym.} }} & \textbf{Linear}    & \makecell*[c]{$\delta_{2r,2r}<\frac15$\\ \citep{ge2016matrix}}  & $\delta_{2r,2r}\leq \frac13$    & \makecell*[c]{$\delta_{2r,2r}<\frac1{10}$\\ \citep{ge2017no}}  & $\delta_{2r,2r} < \frac13$           \\
          & \textbf{Nonlinear}    & \makecell*[c]{$\delta_{2r,4r} < \frac15$\\ \citep{li2019non}}  & $\delta_{2r,2r}\leq \frac13$            & -  & $\delta_{2r,2r}<\frac13$           \\
          \midrule[1pt]
          \multirow{2}{*}{\makecell*[c]{\textbf{Rank-$r$} \\ \textbf{Asym.} }} & \textbf{Linear}    & \makecell*[c]{$\delta_{2r,2r}<\frac13$\\ \citep{ha2020equivalence}}  & $\delta_{2r,2r}\leq \frac13$    & \makecell*[c]{$\delta_{2r,2r}<\frac1{20}$\\ \citep{ge2017no}}  & $\delta_{2r,2r} < \frac13$           \\
          & \textbf{Nonlinear}    & \makecell*[c]{$\delta_{2r,2r}<\frac13$\\ \citep{ha2020equivalence}}  & $\delta_{2r,2r}\leq \frac13$            & \makecell*[c]{$\delta_{2r,4r}<\frac{\alpha(M^*_a)}{100}$\\ \citep{zhu2021global}}  & $\delta_{2r,2r} < \frac13$           \\
          \bottomrule[2pt]
      \end{tabular}
  \end{center}
  \vspace{-2em}
\end{table}

First, we derive necessary conditions and sufficient conditions for the existence of spurious second-order critical points for both symmetric and asymmetric problems. Using our necessary conditions, we show that the $\delta$-RIP$_{2r,2r}$ property with $\delta \leq 1/3$ is enough to guarantee the non-existence of such points. This result provides a marginal improvement to the previous work \citep{ha2020equivalence}, which developed the sufficient condition $\delta<1/3$ for asymmetric problems, and is a major improvement over \cite{ge2016matrix} and \cite{li2019non}, which requires $\delta<1/5$ for symmetric problems. With this non-existence property and under some common regularity conditions, \cite{lee2016gradient,panageas2016gradient} showed that the vanilla gradient descent method with a small enough step size and a random initialization almost surely converges to a global minimum. We note that the convergence rate was not studied and could theoretically be exponential in the worst case. In addition, by studying our necessary conditions, we show that every second-order critical point has a positive correlation to the global minimum when $\delta\in(1/3,1/2)$. When $\delta = 1/2$, a counterexample with spurious second-order critical points is given by utilizing the sufficient conditions. We note that the sufficient conditions can greatly simplify the construction of counterexamples.

Second, we separately study the rank-$1$ case to further strengthen the bounds. In particular, we utilize the necessary conditions to prove that the $\delta$-RIP$_{2,2}$ property with $\delta<1/2$ is enough for the non-existence of spurious second-order critical points. Combining with a counterexample in the $\delta=1/2$ case, we conclude that the bound $\delta<1/2$ is the sharpest bound for the rank-$1$ case. Our results significantly extend the bounds in \cite{zhang2019sharp} derived for the linear symmetric case to the linear asymmetric case and the general nonlinear case. It also improves the bound in \cite{biglobal} by dropping the BDP constant. 

Third, we prove that in the exact parametrization case, problems \eqref{eqn:sym}-\eqref{eqn:asym-reg} both satisfy the strict saddle property \citep{sun2018geometric} when the $\delta$-RIP$_{2r,2r}$ property is satisfied with $\delta<1/3$. This result greatly improves the bounds in \cite{ge2017no,zhu2021global} and extends the result in \cite{ha2020equivalence} to approximate second-order critical points. With the strict saddle property and certain smoothness properties, a wide range of algorithms guarantee a global polynomial-time convergence with a random initialization; see \cite{jin2017escape,jin2018accelerated,sun2018geometric,huang2019perturbed}. Due to the special non-convex structure of our problems and the RIP property, it is possible to prove the boundedness of the trajectory of the perturbed gradient descent method using a similar method as in \cite{jin2017escape}. Since the smoothness properties are satisfied over a bounded region, combining with the strict saddle property, it follows that the perturbed gradient descent method \citep{jin2017escape} achieves a polynomial-time global convergence when $\delta<1/3$.

\subsection{Notation and organization}

The operator $2$-norm and the Frobenius norm of a matrix $M$ are denoted as $\|M\|_2$ and $\|M\|_F$, respectively. The trace of matrix $M$ is denoted as $\mathrm{tr}(M)$. The inner product between two matrices is defined as $\langle M,N\rangle:=\tr(M^TN)$. For any matrix $M\in\mathbb{R}^{n\times m}$, we denote its singular values by $\sigma_1(M) \geq \cdots \geq \sigma_k(M)$, where $k:=\min\{n,m\}$. For any symmetric matrix $M\in\mathbb{R}^{n\times n}$, we denote its eigenvalues by $\lambda_1(M) \geq \cdots \geq \lambda_n(M)$. The minimal eigenvalue is denoted as $\lambda_{min}(\cdot)$. For any matrix $U$, we use $\mathcal{P}_U$ to denote the orthogonal projection onto the column space of $U$. For any matrices $A,B\in\mathbb{R}^{n\times m}$, we use $A\otimes B$ to denote the fourth-order tensor whose $(i,j,k,\ell)$ element is $A_{i,j}B_{k,\ell}$. The identity tensor is denoted as $\mathcal{I}$. The notation $M\succeq 0$ means that the matrix $M$ is symmetric and positive semi-definite. The sub-matrix $R_{i:j,k:\ell}$ consists of the $i$-th to the $j$-th rows and the $k$-th to the $\ell$-th columns of matrix $R$. 

In Section \ref{sec:svp}, the Singular Value Projection algorithm is analyzed as an enlightening example for our main results. Sections \ref{sec:unique} and \ref{sec:strict-saddle} are devoted to the non-existence of spurious second-order critical points and the strict saddle property of the low-rank optimization problem in both symmetric and asymmetric cases, respectively.

\section{Motivating Example: Singular Value Projection Algorithm}
\label{sec:svp}

Before providing theoretical results for problems \eqref{eqn:sym}-\eqref{eqn:asym-reg}, we first consider the Singular Value Projection Method (SVP) algorithm (Algorithm \ref{alg:svp}) as a motivating example, which is proposed in \cite{jain2010guaranteed}. The SVP algorithm is basically the projected gradient method of the original low-rank problems \eqref{eqn:sym-original}-\eqref{eqn:asym-original} via the truncated SVD. For the asymmetric problem \eqref{eqn:asym-original}, the low-rank manifold is 
\[ \mathcal{M}_{asym} := \{ M \in \mathbb{R}^{n\times m}~|~ \rank(M) \leq r \} \]
and the projection is given by only keeping components corresponding to the $r$ largest singular values. For the symmetric problem \eqref{eqn:sym-original}, the low-rank manifold is
\[ \mathcal{M}_{sym} := \{ M \in \mathbb{R}^{n\times n}~|~ \rank(M) \leq r,\quad M^T=M,\quad M\succeq 0 \}. \]
We assume without loss of generality that the gradient $\nabla f(\cdot)$ is symmetric; see Appendix \ref{sec:conditions} for a discussion. The projection is given by only keeping components corresponding to the $r$ largest eigenvalues and dropping all components with negative eigenvalues. Since both low-rank manifolds are non-convex, the projection solution may not be unique and we choose an arbitrary solution when it is not unique. We note that the above projections are orthogonal in the sense that
\[ \| M_+ - M \|_F = \min_{K \in \mathcal{M}}~\|K - M\|_F, \]
where $M_+$ is the projection of a matrix $M$. Henceforth,  $\mathcal M$ stands for $\mathcal{M}_{sym}$ or $\mathcal{M}_{asym}$, which should be clear from the context. Although each truncated SVD operation can be computed within $O(nmr)$ operations, the constant hidden in the $O(\cdot)$ notation is considerably larger than $1$. Thus, the truncated SVD operation is significantly slower than matrix multiplication, which makes the SVP algorithm impractical for large-scale problems. However, the analysis of the SVP algorithm, combining with the equivalence property given in \cite{ha2020equivalence}, provides some insights into how to develop proof techniques for problems \eqref{eqn:sym}-\eqref{eqn:asym-reg}. 

We extend the proof in \cite{jain2010guaranteed} and show that Algorithm \ref{alg:svp} converges linearly to the global minimum under the $\delta$-RIP$_{2r,2r}$ property with $\delta < 1/3$.
\begin{theorem}\label{thm:svp}
If function $f_s(\cdot)$ (resp. $f_a(\cdot)$) satisfies the $\delta$-RIP$_{2r,2r}$ property with $\delta < 1/3$ and the step size is chosen to be $\eta = (1+\delta)^{-1}$, then Algorithm \ref{alg:svp} applied to problem \eqref{eqn:sym-original} (resp. \eqref{eqn:asym-original}) returns a solution $M_T$ such that $M_T \in \mathcal{M}$ and $f(M_T)-f(M^*) \leq \epsilon$ within
\[ T := \left\lceil \frac{1}{\log[(1-\delta) / (2\delta)]} \cdot \log\left[\frac{f(M_0) - f(M^*)}{\epsilon}\right] \right\rceil \]
iterations, where $f(\cdot) := f_s(\cdot)$ (resp. $f(\cdot) := f_a(\cdot)$), $M^*$ is the global minimum, $M_0$ is the initial point and $\lceil\cdot\rceil$ is the ceiling function.
\end{theorem}

We note that the above proof can be applied to other low-rank optimization problems with a suitable definition of the orthogonal projection. In \cite{ha2020equivalence}, it is proved that the unique global minimum is the only fixed point of the SVP algorithm if the RIP constant $\delta$ is less than $1/3$. However, the above paper has not proven the linear convergence (as done in Theorem \ref{thm:svp}). This difference leads to a strengthened inequality in the following analysis, which further serves as an essential step in proving the strict saddle property.

%
\begin{algorithm}[t]
\caption{Singular Value Projection (SVP) Algorithm}
\label{alg:svp}
\begin{algorithmic}[1]
\Require Low-rank manifold $\mathcal{M}$, initial point $M_0$, number of iterations $T$, step size $\eta$, objective function $f(\cdot)$.
\Ensure Low-rank solution $M_T$.
\For{$t=0,\dots,T-1$}
    \State Update $\tilde{M}_{t+1} \leftarrow M_t - \eta \nabla f(M_t)$.
    \State Set $M_{t+1}$ to be the projection of $\tilde{M}_{t+1}$ onto $\mathcal{M}$ via truncated SVD. 
\EndFor
\State \textbf{return} $M_{T}$.
\end{algorithmic}
\end{algorithm}

\section{No Spurious Second-order Critical Points}
\label{sec:unique}


In this section, we develop necessary conditions and sufficient conditions for the existence of spurious second-order critical points of problems \eqref{eqn:sym}-\eqref{eqn:asym-reg}. Besides the non-existence of spurious local minima, the non-existence of spurious second-order critical points also guarantees the global convergence of many first-order algorithms with random initialization under certain regularity conditions \citep{lee2016gradient,panageas2016gradient}. More precisely, we require the iteration points of the algorithm to converge to a single point and the objective function to have a Lipschitz-continuous gradient. The first condition is satisfied by the gradient descent method applied to a large class of functions known as the {K\L}-functions \citep{attouch2013convergence}. For the second condition, many objective functions that appear in applications, e.g., the $\ell_2$-loss function, do not satisfy this condition. However, if the step size is small enough, the special non-convex structure of the Burer-Monterio decomposition and the RIP property ensure that the trajectory of the gradient descent method stays in a compact set, where the Lipschitz condition is satisfied due to the second-order continuity of the functions $f_s(\cdot)$ and $f_a(\cdot)$. The proof of this claim is similar to Theorem 8 in \cite{jin2017escape} and is omitted here. 
Therefore, the non-existence of spurious second-order critical points can ensure the global convergence of the gradient descent method for many applications. 

The non-existence of spurious second-order critical points has been proved in \cite{ge2017no,zhu2018global} for problems with linear and nonlinear measurements, respectively. Recently, \cite{ha2020equivalence} proved a relation between the second-order critical points of problem \eqref{eqn:asym} or \eqref{eqn:asym-reg} and the fixed points of the SVP algorithm on problem \eqref{eqn:asym-original}. Using this relation, they showed that problems \eqref{eqn:asym} and \eqref{eqn:asym-reg} have no spurious second-order critical points when the $\delta$-RIP$_{2r,2r}$ property is satisfied with $\delta < 1/3$. 
In this work, we take a different approach to show that $\delta\leq 1/3$ is enough for the general case in both symmetric and asymmetric scenarios, and that $\delta < 1/2$ is enough for the rank-$1$ case. Moreover, we prove that there exists a positive correlation between every second-order critical point and the global minimum when $\delta\in(1/3,1/2)$. We also show that there may exist spurious second-order critical points when $\delta = 1/2$ for both the symmetric and asymmetric problems, which extends the construction of such examples for the linear symmetric rank-$1$ problem in \cite{zhang2018much} to general cases.
We first give necessary conditions and sufficient conditions for the existence of spurious second-order critical points below.
\begin{theorem}\label{thm:nes-suf}
Let $\ell := \min\{m,n,2r\}$. For a given $\delta \in [0,1)$, there exists a function $f_a(\cdot)$ with the $\delta$-RIP$_{2r,2r}$ property such that problems \eqref{eqn:asym} and \eqref{eqn:asym-reg} have a spurious second-order critical point only if there exist a constant $\alpha \in (1-\delta,(1+\delta)/2]$, a diagonal matrix $\Sigma\in\mathbb{R}^{r\times r}$, a diagonal matrix $\Lambda \in\mathbb{R}^{(\ell-r)\times (\ell-r)}$ and matrices $A\in\mathbb{R}^{r\times r},B\in\mathbb{R}^{r\times r},C\in\mathbb{R}^{(\ell-r)\times r},D\in\mathbb{R}^{(\ell-r)\times r}$ such that
\begin{align}\label{eqn:rank-1-thm}
    \nonumber&(1+\delta) \min_{1\leq i\leq r} \Sigma_{ii} \geq \max_{1\leq i\leq \ell-r} \Lambda_{ii},\quad \Sigma \succ 0,~ \Lambda \succeq 0,\\
    &\langle \Lambda, CD^T \rangle = \alpha \left[ \tr(\Sigma^2) - 2\langle \Sigma, AB^T \rangle + \|AB^T\|_F^2 + \|AD^T\|_F^2 + \|CB^T\|_F^2 + \|CD^T\|_F^2 \right],\\
    \nonumber&\tr(\Lambda^2) \leq \alpha^{-1}( 2\alpha - 1 + \delta^2 ) \cdot \langle \Lambda, CD^T \rangle, \quad \langle \Lambda, CD^T \rangle \neq 0.
\end{align}
If $CB^T=0$ and $AD^T=0$, then there exists a function $f_a(\cdot)$ with the $\delta$-RIP$_{2r,2r}$ property such that problems \eqref{eqn:asym} and \eqref{eqn:asym-reg} have a spurious second-order critical point.
\end{theorem}

The original problem of the non-existence of spurious second-order critical points can be viewed as a property of the set of functions satisfying the RIP property, which is a convex set in an infinite-dimensional functional space. The conditions in \eqref{eqn:rank-1-thm} reduce the infinite-dimensional problem to a finite-dimensional problem by utilizing the optimality conditions and the RIP property, which provides a basis of solving these conditions numerically. We note that the conditions in the third line of \eqref{eqn:rank-1-thm} are novel and serve as an important step in developing strong theoretical guarantees. Although the conditions in \eqref{eqn:rank-1-thm} seem complicated, they lead to strong results on the non-existence of spurious second-order critical points. We provide two corollaries below to illustrate the power of the above theorem. The first corollary focuses on the rank-$1$ case. In this case, we can simplify condition \eqref{eqn:rank-1-thm} through suitable relaxations to obtain a sharper bound on $\delta$ that ensures the non-existence of spurious second-order critical points.
\begin{corollary}\label{cor:rank-1}
Consider the case $r=1$, and suppose that the function $f_a(\cdot)$ satisfies the $\delta$-RIP$_{2,2}$ property with $\delta < 1/2$. Then, problems \eqref{eqn:asym} and \eqref{eqn:asym-reg} have no spurious second-order critical points.
\end{corollary}
The following example shows that the counterexample in \cite{zhang2019sharp} designed for the symmetric case also works for the asymmetric rank-$1$ case.
\begin{example}\label{exm:2}
We note that Example 12 in \cite{zhang2019sharp} shows that problem \eqref{eqn:sym} may have spurious second-order critical points when $\delta = 1/2$. In general, a second-order critical point for problem \eqref{eqn:sym} is not a second-order critical point for problem \eqref{eqn:asym-reg}, since the asymmetric manifold $\mathcal{M}_{asym}$ has a larger second-order critical cone than the symmetric manifold $\mathcal{M}_{sym}$. However, it can be verified that the same example also has a spurious second-order critical point in the asymmetric case. For completeness, we verify the claim in the appendix.
\end{example}
It follows from Corollary \ref{cor:rank-1} and Example \ref{exm:2} that the bound $1/2$ is the \emph{sharpest} bound for the rank-$1$ asymmetric case. The next corollary provides a marginal improvement to the state-of-the-art result for the general rank case, which derives the RIP bound $\delta < 1/3$ \citep{ha2020equivalence}. In addition, we prove that there exists a positive correlation between every second-order critical point and the global minimum when $\delta < 1/2$.
\begin{corollary}\label{cor:rank-r}
Given an arbitrary $r$, suppose that the function $f_a(\cdot)$ satisfies the $\delta$-RIP$_{2r,2r}$ property. If $\delta \leq 1/3$, then both problems \eqref{eqn:asym} and \eqref{eqn:asym-reg} have no spurious second-order critical points. In addition, if $\delta \in[0, 1/2)$, then every second-order critical point $\tilde{M}$ has a positive correlation with the ground truth $M^*_a$. Namely, there exists a universal function $C(\delta) : (0,1/2) \mapsto (0,1]$ such that 
\[ \langle \tilde{M},M^*_a \rangle \geq C(\delta) \cdot \|\tilde{M}\|_F\|M^*_a\|_F. \]
\end{corollary}

For the general rank-$r$ case, we construct a counterexample with spurious second-order critical points when $\delta = 1/2$.
\begin{example}

Let $n = m = 2r$. Now, we use the sufficiency part of Theorem \ref{thm:nes-suf} to construct a counterexample. We choose
\begin{align*}
    \delta := \frac12,\quad \alpha := \frac35,\quad \Sigma := \frac12 I_r,\quad \Lambda := \frac34 I_{r},\quad A = B := 0_r,\quad C = D := I_{r}.
\end{align*}
It can be verified that the conditions in \eqref{eqn:rank-1-thm} are satisfied and $CB^T=AD^T=0$, which means that there exists a function $f_a(\cdot)$ satisfying the $\delta$-RIP$_{2r,2r}$ property for which problems \eqref{eqn:asym} and \eqref{eqn:asym-reg} have spurious second-order critical points. We also give a direct construction with linear measurements in the appendix. This example illustrates that Theorem \ref{thm:nes-suf} can be used to systematically design instances of the problem with spurious second-order critical points.
\end{example}
Before closing this section, we note that similar conditions can be obtained for the symmetric problem \eqref{eqn:sym}. Although there exists a natural transformation of symmetric problems to asymmetric problems (see the appendix), the approach requires the objective function $f_s(\cdot)$ to have the $\delta$-RIP$_{4r,2r}$ property, which provides sub-optimal RIP bounds compared to a direct analysis. We give the results of the direct analysis below and omit the proof due to the similarity to the asymmetric case.
\begin{theorem}
Let $\ell := \min\{n,2r\}$. For a given $\delta \in [0,1)$, there exists a function $f_s(\cdot)$ with the $\delta$-RIP$_{2r,2r}$ property such that problem \eqref{eqn:sym} has a spurious second-order critical point only if there exist a constant $\alpha \in (1-\delta,(1+\delta)/2]$, a diagonal matrix $\Sigma\in\mathbb{R}^{r\times r}$, a diagonal matrix  $ \Lambda \in\mathbb{R}^{(\ell-r)\times (\ell-r)}$ and matrices $A\in\mathbb{R}^{r\times r},C\in\mathbb{R}^{(\ell-r)\times r}$ such that
\begin{align}\label{eqn:rank-1-thm-sym}
    \nonumber&(1+\delta) \min_{1\leq i\leq r} \Sigma_{ii} \geq \max_{1\leq i\leq \ell-r} \Lambda_{ii},\quad \Sigma \succ 0,\\
    &\langle \Lambda, CC^T \rangle = \alpha \left[ \tr(\Sigma^2) - 2\langle \Sigma, AA^T \rangle + \|AA^T\|_F^2 + 2\|AC^T\|_F^2 + \|CC^T\|_F^2 \right],\\
    \nonumber&\tr(\Lambda^2) \leq \alpha^{-1}( 2\alpha - 1 + \delta^2 ) \cdot \langle \Lambda, CC^T \rangle, \quad \langle \Lambda, CC^T \rangle \neq 0.
\end{align}
If $AC^T=0$, then there exists a function $f_s(\cdot)$ with the $\delta$-RIP$_{2r,2r}$ property for which problem \eqref{eqn:sym} has a spurious second-order critical point.
\end{theorem}

Compared to Theorem \ref{thm:nes-suf}, the diagonal matrix $\Lambda$ is not enforced to be positive semi-definite. The reason is that the eigenvalue decomposition is used instead of the singular value decomposition in the symmetric case, and therefore some eigenvalues can be negative. Similarly, we can obtain the non-existence and the positive correlation results for the symmetric problem.
\begin{corollary}
If function $f_s(\cdot)$ satisfies the $\delta$-RIP$_{2r,2r}$ property, then the following statements hold.
\begin{itemize}
    \item If $\delta \leq 1/3$, then there are no spurious second-order critical points;
    \item If $\delta < 1/2$, then there exists a positive correlation between every second-order critical point and the ground truth;
    \item If $\delta = 1/2$, then there exists a counterexample with spurious second-order critical points;
    \item If $\delta < 1/2$ and $r=1$, then there are no spurious second-order critical points.
\end{itemize}
\end{corollary}
We note that the last statement serves as a generalization of the results in \cite{zhang2019sharp} to the nonlinear measurement case, and improves upon the bound in \cite{biglobal} by dropping the BDP constant.

\section{Global Landscape: Strict Saddle Property}
\label{sec:strict-saddle}

Although the non-existence of spurious second-order critical points can ensure the global convergence under certain regularity conditions, it cannot guarantee a fast convergence rate in general. Saddle-point escaping algorithms may become stuck at approximate second-order critical points for exponentially long time. To guarantee the global polynomial-time convergence, the following strict saddle property is commonly considered in the literature.
\begin{definition}[\cite{sun2018geometric}]
Consider an arbitrary optimization problem $\min_{x\in\mathcal{X}\subset\mathbb{R}^d} F(x)$ and let $\mathcal{X}^*$ denote the set of its global minima. It is said that the problem satisfies the $(\alpha,\beta,\gamma)$-\textbf{strict saddle property} for $\alpha,\beta,\gamma > 0$ if at least one of the following conditions is satisfied for every $x\in\mathcal X$:
\[  \mathrm{dist}( x,\mathcal{X}^* ) \leq \alpha;\quad \| \nabla F(x) \|_F \geq \beta;\quad \lambda_{min}[\nabla^2 F(x)] \leq -\gamma. \]
%
%
\end{definition}
For the low-rank problems, we choose the distance to be the Frobenius norm in the factorization space. This distance is equivalent to  the Frobenius norm in the matrix space in the sense that there exist constants $c_1(\mathcal{X}^*)>0$ and $c_2(\mathcal{X}^*)>0$ such that
\[ c_1(\mathcal{X}^*)\cdot \|U - U^*\|_F \leq \|UU^T - U^*(U^*)^T\|_F \leq c_2(\mathcal{X}^*) \cdot \|U-U^*\|_F \]
holds for all $U \in \mathcal X$ as long as $\|U-U^*\|_F$ is small \citep{tu2016low}. A similar relation holds for the asymmetric case.

In \cite{jin2017escape}, it has been proved that the perturbed gradient decent method can find an $\epsilon$-approximate second-order critical point in $\tilde{O}(\epsilon^{-2})$ iterations with high probability if the Hessian of the objective function is Lipschitz. Namely, the algorithm returns a point $x\in\mathcal{X}$ such that
\[ \| \nabla F(x) \|_F \leq O(\epsilon), \quad \lambda_{min}[\nabla^2 F(x)] \geq -O(\sqrt{\epsilon})  \]
in $\tilde{O}(\epsilon^{-2})$ iterations with high probability. If we choose $\epsilon > 0$ to be small enough such that $O(\epsilon) < \beta$ and $-O(\sqrt{\epsilon}) > - \gamma$, then the strict saddle property ensures that the returned point satisfies $\mathrm{dist}( x,\mathcal{X}^* ) \leq \alpha$ with high probability. We note that the Lipschitz continuity of the Hessian can be similarly guaranteed by the boundedness of trajectories of the perturbed gradient method, which can be proved similarly as Theorem 8 in \cite{jin2017escape}. 
Since the smoothness properties are satisfied over a bounded region, we may apply the perturbed gradient descent method \citep{jin2017escape} to achieve the polynomial-time global convergence with random initialization.

In this section, we prove that problems \eqref{eqn:sym} and \eqref{eqn:asym-reg} satisfy the strict saddle property with an arbitrary $\alpha >0$ in the exact parameterization case, i.e., when the global optimum has rank $r$.
\begin{assumption}\label{asp:2}
The global optimum $M^*_a$ or $M^*_s$ has rank $r$.
\end{assumption}

It has been proved in \cite{zhu2021global} that the regularized problem \eqref{eqn:asym-reg} satisfies the strict saddle property if the function $f_a(\cdot)$ has the $\delta$-RIP$_{2r,4r}$ property with
\[ \delta < \frac{\sigma_r(M^*_a)^{3/2}}{100\|M^*_a\|_F\|M^*_a\|_2^{1/2}}. \]
Our results improve upon their bounds by allowing a larger problem-free RIP constant and requiring only the RIP$_{2r,2r}$ property (note that there are problems with RIP$_{2r,2r}$ property while RIP$_{2r,4r}$ property does not hold \citep{biglobal}). Our result can also be viewed as a robust version of the results in \cite{ha2020equivalence}.

%
\begin{theorem}\label{thm:strict-saddle-1}
Suppose that the function $f_a(\cdot)$ satisfies the $\delta$-RIP$_{2r,2r}$ property with $\delta < 1/3$. Given an arbitrary constant $\alpha>0$, if $\mu$ is selected to belong to the interval $[(1-\delta)/3,1-\delta)$, then there exist positive constants
\[ \epsilon_1 := \epsilon_1(\delta,r,\mu,\sigma_r(M^*_a),\|M^*_a\|_F,\alpha),\quad \lambda_1 := \lambda_1(\delta,r,\mu,\sigma_r(M^*_a),\|M^*_a\|_F,\alpha) \]
such that for every $\epsilon\in(0,\epsilon_1]$ and $\lambda\in(0,\lambda_1]$, 
problem \eqref{eqn:asym-reg} satisfies the $(\alpha,\beta,\gamma)$-strict saddle property with
\[ \beta := \min\left\{ \mu(\epsilon/r)^{3/2}, \lambda \right\} ,\quad \gamma := \mu\epsilon. \]
%
\end{theorem}
We note that the constraint $\mu\in[(1-\delta)/3,1-\delta)$ is not optimal and it can be similarly proved that $\mu\in(\delta,1-\delta)$ also guarantees the strict saddle property. The key step in the proof is to show that for every point $(U,V)$ at which the gradient of $f_a(UV^T)$ is small,  it holds that
\[ \|\nabla f_a(UV^T)\|_2 \geq (1+\delta)\sigma_r(UV^T) + C \cdot  (1-3\delta) [ f_a(UV^T) - f_a(M^*_a) ], \]
where $C > 0$ is a constant independent of $(U,V)$. This inequality can be viewed as a major extension of the non-existence of spurious second-order critical points when $\delta < 1/3$ \citep{ha2020equivalence}, which shows that every spurious second-order critical point $(U,V)$ satisfies
\[ \|\nabla f_a(UV^T)\|_2 > (1+\delta)\sigma_r(UV^T). \]
We emphasize that our proof requires a new framework and is not a standard revision of the existing methods, which is the reason why sharper bounds can be established. By replacing $\|\nabla f_a(M)\|_2$ with $-\lambda_{min}(\nabla f_s(M))$, the analysis for the asymmetric case can be extended to the symmetric case with minor modifications and the same bound follows.
\begin{theorem}\label{thm:strict-saddle-2}
Suppose that the function $f_s(\cdot)$ satisfies the $\delta$-RIP$_{2r,2r}$ property with $\delta < 1/3$. Given an arbitrary constant $\alpha>0$, there exists a positive constant $\lambda_1 := \lambda_1(\delta,r,\sigma_r(M^*_s),\|M^*_s\|_F,\alpha)$ such that for every $\lambda\in(0,\lambda_1]$, 
problem \eqref{eqn:sym} satisfies the $(\alpha,\beta,\gamma)$-strict saddle property with
\[ \beta :=  \lambda  ,\quad \gamma := 2\lambda. \]
%
\end{theorem}
The above bound is the first theoretical guarantee of the strict saddle property for the nonlinear symmetric problem.

\section{Conclusion}

In this work, we analyze the geometric properties of low-rank optimization problems via the non-convex factorization approach. We prove novel necessary conditions and sufficient conditions for the non-existence of spurious second-order critical points in both symmetric and asymmetric cases. We show that these conditions lead to sharper bounds and greatly simplify the construction of counterexamples needed to study the sharpness of the bounds. The developed bounds significantly generalize several of the existing results. In the rank-$1$ case, the bound is proved to be the sharpest possible. In the general rank case, we show that there exists a positive correlation between second-order critical points and the global minimum for problems whose RIP constants are higher than the developed bound but lower than the fundamental limit obtained by the counterexamples. Finally, the strict saddle property is proved with a weaker requirement on the RIP constant for asymmetric problems. The paper develops the first strict saddle property in the literature for nonlinear symmetric problems.  


\bibliographystyle{iclr2020_conference}
\bibliography{reference}

\newpage
\begin{appendix}

\section{Optimality Conditions}\label{sec:conditions}

In this section, we develop the optimality conditions for problems \eqref{eqn:sym-original}-\eqref{eqn:asym-reg}. We assume without loss of generality that $\nabla f_s(M)$ is symmetric for every $M\in\mathbb{R}^{n\times n}$. This is because we can always optimize the equivalent problem
\[ \min_{M\in\mathbb{R}^{n\times n}}~ \frac12\left[f_s(M) + f_s(M^T)\right] \quad \mathrm{s.t.}~ \rank(M) \leq r,\quad M^T = M,\quad M\succeq0. \]
We first consider problems \eqref{eqn:sym-original} and \eqref{eqn:asym-original}.
\begin{theorem}[\cite{li2019non,ha2020equivalence}]
The matrix $\tilde{M} = \tilde{U}\tilde{U}^T$ with $\tilde{U}\in\mathbb{R}^{n\times r}$ is a first-order critical point of problem \eqref{eqn:sym-original} if and only if
\begin{align*}
    \begin{cases} \nabla f_s(\tilde{M})\tilde{U} = 0 &\text{if }\rank(\tilde{M}) = r\\ 
    \nabla f_s(\tilde{M}) \succeq 0 &\text{if }\rank(\tilde{M}) < r. \end{cases}
\end{align*}
The matrix $\tilde{M} = \tilde{U}\tilde{V}^T$ with $\tilde{U}\in\mathbb{R}^{n\times r}$ and $\tilde{V}\in\mathbb{R}^{m\times r}$ is a first-order critical point of problem \eqref{eqn:asym-original} if and only if
\begin{align*}
    \begin{cases}[\nabla f_a(\tilde{M})]^T\tilde{U} = 0, ~ \nabla f_a(\tilde{M}) \tilde{V} = 0 &\text{if }\rank(\tilde{M}) = r\\
    \nabla f_a(\tilde{M}) = 0 &\text{if }\rank(\tilde{M}) < r.\end{cases}
\end{align*}
\end{theorem}
In \cite{ha2020equivalence}, the authors proved that each second-order critical point of problem \eqref{eqn:asym} or \eqref{eqn:asym-reg} is a fixed point of the SVP algorithm run on problem \eqref{eqn:asym-original}. We note that this relation can be extended to the symmetric and positive semi-definite case. This relation plays an important role in the analysis of Section \ref{sec:unique}. 
\begin{theorem}[\cite{ha2020equivalence}]\label{thm:fixed-pt}
The matrix $\tilde{M} = \tilde{U}\tilde{U}^T$ with $\tilde{U}\in\mathbb{R}^{n\times r}$ is a fixed point of the SVP algorithm run on problem \eqref{eqn:sym-original} with the step size $1/(1+\delta)$ if and only if
\begin{align*}
    \nabla f_s(\tilde{M})\tilde{U} = 0, \quad -\lambda_{min}(\nabla f_s(\tilde{M})) \leq (1+\delta) \sigma_r(\tilde{U}).
\end{align*}
The matrix $\tilde{M} = \tilde{U}\tilde{V}^T$ with $\tilde{U}\in\mathbb{R}^{n\times r}$ and $\tilde{V}\in\mathbb{R}^{m\times r}$ is a fixed point of the SVP algorithm run on problem \eqref{eqn:asym-original} with the step size $1/(1+\delta)$ if and only if
\begin{align*}
    [\nabla f_a(\tilde{M})]^T\tilde{U} = 0, \quad \nabla f_a(\tilde{M}) \tilde{V} = 0,\quad \|\nabla f_a(\tilde{M})\|_2 \leq (1+\delta) \sigma_r(\tilde{M}).
\end{align*}
\end{theorem}
Next, we consider problems \eqref{eqn:asym}-\eqref{eqn:asym-reg}. the goal is to study only spurious local minima and saddle points, it is enough to focus on the second-order necessary optimality conditions. The following two theorems follow from basic calculations and we omit the proof.
\begin{theorem}
The matrix $\tilde{U}\in\mathbb{R}^{n\times r}$ is a second-order critical point of problem \eqref{eqn:sym} if and only if
\begin{align*}
    \nabla f_s(\tilde{U}\tilde{U}^T)\tilde{U} = 0
\end{align*}
and
\begin{align*}
    2\langle \nabla f_s(\tilde{U}\tilde{U}^T), \Delta\Delta^T \rangle + [\nabla^2 f_s(\tilde{U}\tilde{U}^T)]( \tilde{U}\Delta^T + \Delta \tilde{U}^T, \tilde{U}\Delta^T + \Delta \tilde{U}^T ) \geq 0
\end{align*}
hold for every $\Delta \in \mathbb{R}^{n\times r}$. 
\end{theorem}
\begin{theorem}\label{thm:asym-condition}
The point $(\tilde{U},\tilde{V})$ with $\tilde{U}\in\mathbb{R}^{n\times r}$ and $\tilde{V}\in\mathbb{R}^{m\times r}$ is a second-order critical point of problem \eqref{eqn:asym} if and only if
\begin{align*}
    \nabla [f_a(\tilde{U}\tilde{V}^T)]^T \tilde{U} = 0,\quad \nabla f_a(\tilde{U}\tilde{V}^T)\tilde{V} = 0
\end{align*}
and
\begin{align*}
    2\langle \nabla f_a(\tilde{U}\tilde{V}^T), \Delta_U\Delta_V^T \rangle + [\nabla^2 f_a(\tilde{U}\tilde{V}^T)]( \tilde{U}\Delta_V^T + \Delta_U \tilde{V}^T, \tilde{U}\Delta_V^T + \Delta_U \tilde{V}^T ) \geq 0
\end{align*}
hold for every $\Delta_U \in \mathbb{R}^{n\times r}$ and $ \Delta_V \in \mathbb{R}^{m\times r}$. Moreover, the given point is a a second-order critical point of problem \eqref{eqn:asym-reg} if and only if 
\begin{align*}
    \nabla [f_a(\tilde{U}\tilde{V}^T)]^T \tilde{U} = 0,\quad \nabla f_a(\tilde{U}\tilde{V}^T)\tilde{V} = 0,\quad \tilde{U}^T\tilde{U} = \tilde{V}^T\tilde{V}
\end{align*}
and
\begin{align*}
    &2\langle \nabla f_a(\tilde{U}\tilde{V}^T), \Delta_U\Delta_V^T \rangle + [\nabla^2 f_a(\tilde{U}\tilde{V}^T)]( \tilde{U}\Delta_V^T + \Delta_U \tilde{V}^T, \tilde{U}\Delta_V^T + \Delta_U \tilde{V}^T )\\
    &\hspace{16em} + \frac{\mu}{2} \|\tilde{U}^T\Delta_U + \Delta_U^T\tilde{U} - \tilde{V}^T\Delta_V - \Delta_V^T \tilde{V} \|_F^2 \geq 0
\end{align*}
hold for every $\Delta_U \in \mathbb{R}^{n\times r}$ and $ \Delta_V \in \mathbb{R}^{m\times r}$.
\end{theorem}

\section{Relation between the Symmetric and Asymmetric Problems}
\label{sec:relation}

In this section, we study the relationship between problems \eqref{eqn:sym}-\eqref{eqn:asym-reg}. This relationship is more general than the topic of this paper, namely the non-existence of spurious second-order critical points and the strict saddle property, and holds for any property that is characterized by the RIP constant $\delta$ and the BDP constant $\kappa$. Specifically, we show that any property that holds for the symmetric problems \eqref{eqn:sym} with $(\delta,\kappa)$ also holds for the regularized asymmetric problem \eqref{eqn:asym-reg} with another pair of constants $(\tilde{\delta},\tilde{\kappa})$ decided by $\delta,\kappa$, and vice versa.

We first consider the transformation from the asymmetric case to the symmetric case. The transformation to the symmetric case has been established in \cite{ge2017no} for linear problem. Here, we show that the transformation can be revised and extended to the nonlinear measurements case.
\begin{theorem}\label{thm:reduction}
Suppose that the function $f_a(\cdot)$ satisfies the $\delta$-RIP$_{2r,2s}$ and the $\kappa$-BDP$_{2t}$ properties. If we choose $\mu := (1-\delta)/2$, then problem \eqref{eqn:asym-reg} is equivalent to a symmetric problem whose objective function satisfies the $2\delta/(1+\delta)$-RIP$_{2r,2s}$ and the $2\kappa/(1+\delta)$-BDP$_{2t}$ properties.
\end{theorem}
\begin{proof}[Proof of Theorem \ref{thm:reduction}]
For any matrix $N\in\mathbb{R}^{(n+m)\times(n+m)}$, we divide the matrix into four blocks as
\[ N = \begin{bmatrix} N_{11} & N_{12} \\ N_{21} & N_{22} \end{bmatrix}, \]
where $N_{11}\in\mathbb{R}^{n\times n},N_{12}\in\mathbb{R}^{n\times m},N_{22}\in\mathbb{R}^{m\times m}$. Then, we define a new function
\[ \tilde{f}(N) := f_a(N_{12}) + f_a(N_{21}^T). \]
We observe that $\tilde{f}(WW^T) = 2h_a(U,V)$, where
\[ W := \begin{bmatrix} U \\ V \end{bmatrix} \in \mathbb{R}^{(n+m) \times r}. \]
For any $K\in\mathbb{R}^{(n+m)\times(n+m)}$, the Hessian of $\tilde{f}(\cdot)$ satisfies
\begin{align}\label{eqn:2} [\nabla^2 \tilde{f}(N)](K,K) = [\nabla^2 {f}_a(N_{12})](K_{12},K_{12}) + [\nabla^2 {f}_a(N_{21}^T)](K_{21}^T,K_{21}^T). \end{align}
Similarly, we can define
\[ \tilde{g}(N) := \|N_{11}\|_F^2 + \|N_{22}\|_F^2 - \|N_{12}\|_F^2 - \|N_{21}\|_F^2. \]
We can also verify that $\tilde{g}(WW^T) = g(U,V)$ and
\begin{align}\label{eqn:3} [\nabla^2 \tilde{g}(N)](K,K) = 2\left( \|K_{11}\|_F^2 + \|K_{22}\|_F^2 - \|K_{12}\|_F^2 - \|K_{21}\|_F^2 \right). \end{align}
for every $K\in\mathbb{R}^{(n+m)\times(n+m)}$. The minimization problem \eqref{eqn:asym-reg} is then equivalent to
\begin{align}\label{eqn:asym-sym}
    \min_{W\in\mathbb{R}^{(n+m)\times r}}~ F(WW^T) := \tilde{f}(WW^T) + \frac{\mu}{2}\cdot \tilde{g}(WW^T),
\end{align}
which is in the symmetric form as problem \eqref{eqn:sym}. For every $N,K\in\mathbb{R}^{(n+m)\times(n+m)}$ with $\rank(N) \leq 2r$ and $\rank(K)\leq 2s$, it results from relations \eqref{eqn:2} and \eqref{eqn:3} that
\begin{align*}
    [\nabla^2 F(N)]&(K,K)\\
    &\geq (1-\delta)\left( \|K_{12}\|_F^2 + \|K_{21}\|_F^2\right) + \mu\left( \|K_{11}\|_F^2 + \|K_{22}\|_F^2 -\|K_{12}\|_F^2 -\|K_{21}\|_F^2 \right)\\
    &\geq \min\{ 1-\delta-\mu, \mu \} \cdot \|K\|_{F}^2
\end{align*}
and
\begin{align*}
    [\nabla^2 F(N)]&(K,K)\\
    &\leq  (1+\delta)\left( \|K_{12}\|_F^2 + \|K_{21}\|_F^2\right) + \mu\left( \|K_{11}\|_F^2 + \|K_{22}\|_F^2 -\|K_{12}\|_F^2 -\|K_{21}\|_F^2 \right)\\
    &\leq \max\{ 1+\delta-\mu, \mu \} \cdot \|K\|_{F}^2.
\end{align*}
Choosing $\mu := (1-\delta)/2$, we obtain
\[ \frac{1-\delta}{2}\cdot \|K\|_{F}^2 \leq [\nabla^2 F(N)](K,K) \leq \frac{1+3\delta}{2} \cdot \|K\|_{F}^2. \]
Hence, it follows that the function $ 2F(\cdot) / (1+\delta)$ satisfies the $2\delta/(1+\delta)$-RIP$_{2r,2s}$ property. 

Moreover, for every $N,N',K,L\in\mathbb{R}^{(n+m)\times(n+m)}$ with
\[ \rank(N),\rank(N'), \rank(K),\rank(L) \leq 2t, \]
it holds that
\begin{align*}
    [\nabla^2 \tilde{g}(N)](K,L) &= [\nabla^2 \tilde{g}(N')](K,L)\\
    &= 2\left( \langle K_{11}, L_{11}\rangle + \langle K_{22}, L_{22}\rangle - \langle K_{12}, L_{12}\rangle - \langle K_{21}, L_{21}\rangle \right)
\end{align*}
and
\begin{align*}
    &\hspace{-4em}\left| [\nabla^2 F(N) - \nabla^2 F(N')]( K,L ) \right|\\
    =& \left| [\nabla^2 {f}(N_{12}) -\nabla^2 {f}(N_{12}') ](K_{12},L_{12}) + [\nabla^2 {f}(N_{21}^T) -\nabla^2 {f}((N_{21}')^T) ](K_{21}^T,L_{21}^T) \right|\\
    \leq& \kappa \|K_{12}\|_{F}\|L_{12}\|_F + \kappa \|K_{21}\|_F\|L_{21}\|_F \leq \kappa \|K\|_F\|L\|_F,
\end{align*}
which implies that the function $\frac{2}{1+\delta} \cdot F(\cdot)$ satisfies the $2\kappa/(1+\delta)$-BDP$_{2r}$ property. Since problem \eqref{eqn:asym-sym} is equivalent to the minimization of $\frac{2}{1+\delta} \cdot F(WW^T)$, it is equivalent to a symmetric problem that satisfies the $2\delta/(1+\delta)$-RIP$_{2r,2s}$ and the $2\kappa/(1+\delta)$-BDP$_{2r}$ properties.
\end{proof}

We can see that both constants $\delta$ and $\kappa$ are approximately doubled in the transformation. As an example, \cite{bhojanapalli2016global} showed that the symmetric linear problem has no spurious local minima if the $\delta$-RIP$_{2r}$ property is satisfied with $\delta < 1/5$. Using Theorem \ref{thm:reduction}, we know that the asymmetric linear problem has no spurious local minima if the $\delta$-RIP$_{2r}$ property is satisfied with $\delta < 1/9$. 

The transformation from a symmetric problem to an asymmetric problem is more straightforward. We can equivalently solve the optimization problem
\begin{align}\label{eqn:relation-1}
    \min_{U,V\in\mathbb{R}^{n\times r}}~f_s\left[ \frac12 \left( UV^T + VU^T \right) \right]
\end{align}
or its regularized version with any parameter $\mu >0$. It can be easily shown that the above problem has the same RIP and BDP constants as the original symmetric problem. We omit the proof for brevity. 
\begin{theorem}\label{thm:reduction-2}
Suppose that the function $f_s(\cdot)$ satisfies the $\delta$-RIP$_{4r,2s}$ and the $\kappa$-BDP$_{4t}$ properties. For every $\mu > 0$, problem \eqref{eqn:sym} is equivalent to an asymmetric problem and its regularized version with the $\delta$-RIP$_{2r,2s}$ and the $\kappa$-BDP$_{2t}$ properties.
\end{theorem}
Note that the transformation from a symmetric problem to an asymmetric problem will not increase the constants $\kappa$ and $\delta$ but requires stronger RIP and BDP properties. Hence, a direct analysis on the symmetric case may establish the same property under a weaker condition. In addition to problem \eqref{eqn:relation-1}, we can also directly consider the problem $\min_{U,V} f_a(UV^T)$. However, in certain applications, the objective function is only defined for symmetric matrices and we can only use the formulation \eqref{eqn:relation-1} to construct an asymmetric problem. In more restricted cases when the objective function is only defined for symmetric and positive semi-definite matrices, we can only apply the direct analysis to the symmetric case.

\section{Proofs for Section \ref{sec:svp}}

\subsection{Proof of Theorem \ref{thm:svp}}

\begin{proof}[Proof of Theorem \ref{thm:svp}]

We denote $f(\cdot):=f_s(\cdot)$ and $f(\cdot):=f_a(\cdot)$ for the symmetric and asymmetric case, respectively. Using the mean value theorem and the $\delta$-RIP$_{2r,2r}$ property, there exists a constant $s\in[0,1]$ such that
\begin{align*}
    & \hspace{1.5em} f(M_{t+1}) - f(M_t)\\
    & = \langle \nabla f(M_t), M_{t+1} - M_t \rangle + \frac{1}{2} [\nabla^2 f(M_t + s(M_{t+1} - M_t))](M_{t+1} - M_t,M_{t+1} - M_t) \\
    &\leq \langle \nabla f(M_t), M_{t+1} - M_t \rangle + \frac{1+\delta}{2} \|M_{t+1} - M_t\|_F^2.
\end{align*}
We define
\[ \phi_t(M) := \langle \nabla f(M_t), M - M_t \rangle + \frac{1+\delta}{2} \|M - M_t\|_F^2 = \frac{1+\delta}{2}\|M - \tilde{M}_{t+1}\|_F^2 + constant, \]
where the last constant term is independent of $M$. Since the projection is orthogonal, the projected matrix $M_{t+1}$ achieves the minimal value of $\phi_t(M)$ over all matrices on the manifold $\mathcal{M}$. Therefore, we obtain
\begin{align}\label{eqn:1}
    \nonumber f(M_{t+1}) - f(M_t) &\leq \phi_t(M_{t+1}) \leq \phi_t(M^*) \\
    &= \langle \nabla f(M_t), M^* - M_t \rangle + \frac{1+\delta}{2} \|M^* - M_t\|_F^2.
\end{align}
On the other hand, we can similarly prove that the $\delta$-RIP$_{2r,2r}$ property ensures
\begin{align*}
f(M^*) - f(M_t) &\geq \langle \nabla f(M_t), M^* - M_t \rangle + \frac{1-\delta}{2} \|M^* - M_t\|_F^2, \\
f(M_t) - f(M^*) &\geq \frac{1-\delta}{2} \|M^* - M_t\|_F^2.
\end{align*}
Substituting the above two inequalities into \eqref{eqn:1}, it follows that
\begin{align}\label{eqn:2-1} 
\nonumber f(M_{t+1}) - f(M_t) &\leq  f(M^*) - f(M_t) + \delta \|M^* - M_t\|_F^2\\
&\leq f(M^*) - f(M_t) + \frac{2\delta}{1-\delta} [f(M_t) - f(M^*)].
\end{align}
Therefore, using the condition that $\delta < 1/3$, we have
\[ f(M_{t+1}) - f(M^*) \leq \frac{2\delta}{1-\delta} [f(M_t) - f(M^*)] := \alpha [f(M_t) - f(M^*)], \]
where $\alpha:= (1-\delta)/(2\delta) < 1$. Combining this single-step bound with the induction method proves the linear convergence of Algorithm \ref{alg:svp}.
\end{proof}

\section{Proofs for Section \ref{sec:unique}}

\subsection{Proof of Theorem \ref{thm:nes-suf}}

\begin{proof}[Proof of Theorem \ref{thm:nes-suf}]

We only consider the case when $m$ and $n$ are at least $2r$. In this case, we have $\ell = 2r$. Other cases can be handled similarly. For the notational simplicity, we denote $M^* := M^*_a$ in this proof.

\paragraph{Necessity.} 
We first consider problem \eqref{eqn:asym}. Suppose that $M^*$ and $\tilde{M}$ are the optimum and a spurious second-order critical point of problem \eqref{eqn:asym}, respectively. It has been proved in \cite{ha2020equivalence} that spurious second-order critical point $\tilde{M}$ has rank $r$ and is a fixed point of the SVP algorithm with the step size $(1+\delta)^{-1}$. Therefore, the point $\tilde{M}$ should be a minimizer of the projection step of the SVP algorithm. This implies that
\begin{align*}
    \| \tilde{M} - [\tilde{M} - (1+\delta)^{-1} \nabla f_a(\tilde{M})] \|_F^2 \leq \| M^* - [\tilde{M} - (1+\delta)^{-1} \nabla f_a(\tilde{M})] \|_F^2,
\end{align*}
which can be simplified to
\begin{align}\label{eqn:rank-1-add-2}
    \langle \nabla f_a(\tilde{M}), \tilde{M} - M^* \rangle \leq \frac{1+\delta}{2}\|\tilde{M} - M^*\|_F^2.
\end{align}
Let $\mathcal{U}$ and $\mathcal{V}$ denote the subspaces spanned by the columns and rows of $\tilde{M}$ and $M^*$, respectively. Namely, we have
\[ \mathcal{U} := \{ \tilde{M} v_1 + M^* v_2 ~|~v_1,v_2 \in \mathbb{R}^m \},\quad \mathcal{V} := \{ \tilde{M}^T u_1 + (M^*)^T u_2 ~|~u_1,u_2 \in \mathbb{R}^n \}. \]
Since the ranks of both matrices are bounded by $r$, the dimensions of $\mathcal{U}$ and $\mathcal{V}$ are bounded by $2r$. Therefore, we can find orthogonal matrices $U\in\mathbb{R}^{n\times 2r}$ and $V\in\mathbb{R}^{m\times 2r}$ such that
\[ \mathcal{U} \subset \mathrm{range}(U),\quad \mathcal{V} \subset \mathrm{range}(V) \]
and write $\tilde{M},M^*$ in the form
\begin{align*}
    \tilde{M} = U \begin{bmatrix} \Sigma & 0_{r\times r}\\ 0_{r\times r} & 0_{r\times r}\\  \end{bmatrix} V^T,\quad M^* = URV^T,
\end{align*}
where $\Sigma\in\mathbb{R}^{r\times r}$ is a diagonal matrix and $R\in\mathbb{R}^{2r\times 2r}$ has rank at most $r$. Recalling the first condition in Theorem \ref{thm:asym-condition}, the column space and the row space of $\nabla f_a(\tilde{M})$ are orthogonal to the column space and the row space of $\tilde{M}$, respectively. Then, the $\delta$-RIP$_{2r,2r}$ property gives
\begin{align}\label{eqn:rank-1-add} 
\nonumber\exists \alpha \in [1-\delta,1+\delta] \quad\mathrm{s.t.}~ &-\langle \nabla f_a(\tilde{M}), M^*\rangle = \langle \nabla f_a(\tilde{M}), \tilde{M} - M^*\rangle\\
\nonumber&= \int_0^1 [\nabla^2 f_a( M^* + s(\tilde{M}-M^*) )]( \tilde{M} - M^*,\tilde{M} - M^* )~ds\\
&= \alpha \|\tilde{M} - M^*\|_F^2 > 0. \end{align}
This means that
\[ G := \mathcal{P}_{U} \nabla f_a(\tilde{M}) \mathcal{P}_V \neq 0, \]
where $\mathcal{P}_{U}$ and $\mathcal{P}_{V}$ are the orthogonal projections onto $\mathcal{U}$ and $\mathcal{V}$, respectively. Combining with inequality \eqref{eqn:rank-1-add-2}, we obtain $\alpha \leq (1+\delta)/2$. By the definition of $G$, we have
\[ \langle \nabla f_a(\tilde{M}), M^*\rangle = \langle G, M^*\rangle. \]
Since both the column space and the row space of $G$ are orthogonal to $\tilde{M}$, the matrix $G$ has the form
\begin{align}\label{eqn:rank-1-2} G = U \begin{bmatrix} 0_{r\times r} & 0_{r\times r}\\ 0_{r\times r} & -\Lambda\\  \end{bmatrix} V^T, \end{align}
where $\Lambda \in\mathbb{R}^{r\times r}$. We may assume without loss of generality that $\Lambda_{ii}\geq0$ for all $i$; otherwise, one can flip the sign of some of the last $r$ columns of $U$. By another orthogonal transformation, we may assume without loss of generality that $\Lambda$ is a diagonal matrix. Then, Theorem \ref{thm:fixed-pt} gives
\begin{align}\label{eqn:rank-1-1}
    (1+\delta) \min_{1\leq i\leq r} \Sigma_{ii} = (1+\delta) \sigma_r(\tilde{M}) \geq \|\nabla f_a(\tilde{M})\|_2 \geq \|G\|_2 = \max_{1\leq i\leq (\ell-r)} \Lambda_{ii}.
\end{align}
In addition, condition \eqref{eqn:rank-1-add} is equivalent to
\begin{align}\label{eqn:rank-1-3}
    \langle \Lambda, R_{r+1:2r, r+1:2r} \rangle = \alpha \|\tilde{M}-M^*\|_F^2 = \alpha \left[ \tr(\Sigma^2) - 2\langle \Sigma, R_{1:r,1:r}\rangle + \|R\|_F^2 \right].
\end{align}
By the Taylor expansion, for every $Z\in\mathbb{R}^{n\times m}$, we have
\[ \langle \nabla f_a(\tilde{M}), Z\rangle = \int_0^1 [\nabla^2 f_a(M^* + s(\tilde{M} - M^*))]( \tilde{M} - M^*, Z )~ds = (\tilde{M} - M^*) : \mathcal{H} : Z, \]
where the last expression is the tensor multiplication and $\mathcal{H}$ is the tensor such that
\[ K : \mathcal{H} :L = \int_0^1 [\nabla^2 f_a(M^* + s(\tilde{M} - M^*))]( K,L )~ds, \quad \forall K,L\in\mathbb{R}^{n\times m}. \] 
We define
\[ \tilde{G} := G - \alpha(\tilde{M} - M^*). \]
By the definition of $\alpha$, we know that $\langle \tilde{G},\tilde{M}-M^*\rangle = 0$. Furthermore, using the definition of $\mathcal{H}$, we obtain
\begin{align*}
    &(\tilde{M} - M^*) : \mathcal{H} : (\tilde{M}-M^*) = \alpha \|\tilde{M} - M^*\|_F^2,\\
    &(\tilde{M} - M^*) : \mathcal{H} : \tilde{G} = \tilde{G} : \mathcal{H} : (\tilde{M}-M^*) = \|\tilde{G}\|_F^2.
\end{align*}
Suppose that
\[ \tilde{G} : \mathcal{H} : \tilde{G} = \beta \|\tilde{G}\|_F^2 \]
for some $\beta \in [1-\delta,1+\delta]$. We consider matrices of the form
\[ K(t) := t (\tilde{M}-M^*) + \tilde{G}, \quad \forall t\in\mathbb{R}. \]
Since $K(t)$ is a linear combination of $\tilde{M}-M^*$ and $G$, the column space of $K(t)$ is a subspace of $\mathcal{U}$, and thus $K(t)$ has rank at most $2r$ and the $\delta$-RIP$_{2r,2r}$ property implies
\begin{align}\label{eqn:rank-1-4} (1-\delta) \|K(t)\|_F^2 \leq  K(t) : \mathcal{H} : K(t) \leq (1+\delta) \|K(t)\|_F^2 . \end{align}
Using the facts that
\begin{align*}
    \|K(t)\|_F^2 &= \|\tilde{M}-M^*\|_F^2\cdot t^2 + \|\tilde{G}\|_F^2,\\
    K(t) : \mathcal{H} : K(t) &= \alpha \|\tilde{M}-M^*\|_F^2 \cdot t^2 + 2 \|\tilde{G}\|_F^2 \cdot t + \beta \|\tilde{G}\|_F^2,
\end{align*}
we can write the two inequalities in \eqref{eqn:rank-1-4} as quadratic inequalities 
\begin{align}\label{eqn:rank-1-4-add}
\nonumber&[\alpha - (1-\delta)] \|\tilde{M}-M^*\|_F^2 \cdot t^2 + 2 \|\tilde{G}\|_F^2 \cdot t+ [\beta - (1-\delta)] \|\tilde{G}\|_F^2 \geq0,\\
&[(1+\delta) - \alpha] \|\tilde{M}-M^*\|_F^2 \cdot t^2 - 2 \|\tilde{G}\|_F^2 \cdot t + [(1+\delta) - \beta] \|\tilde{G}\|_F^2 \geq0.
\end{align}
If $\alpha = 1-\delta$, then we must have $\|\tilde{G}\|_F = 0$ and thus $G = \alpha(\tilde{M} - M^*)$. Equivalently, we have $M^* = \tilde{M} - \alpha^{-1} G$. Since the column and row spaces of $G \neq0$ are orthogonal to $\tilde{M}$, the rank of $M^*$ is at least $\rank(\tilde{M}) + 1 = r+1$, which is a contradiction. 
Since $\alpha \leq (1+\delta) / 2$, we have $\alpha < 1+\delta$. Thus, we have proved that
\[ 1-\delta < \alpha < 1+\delta. \]
Checking the condition for quadratic functions to be non-negative, we obtain
\begin{align*}
    \|\tilde{G}\|_F^2 &\leq [\alpha - (1-\delta)][\beta - (1-\delta)] \cdot \|\tilde{M}-M^*\|_F^2,\\
    \|\tilde{G}\|_F^2 &\leq [(1+\delta) - \alpha][(1+\delta) - \beta] \cdot \|\tilde{M}-M^*\|_F^2.
\end{align*}
Since
\[ \alpha - (1-\delta) > 0,\quad (1+\delta) - \alpha > 0, \]
the above two inequalities are equivalent to
\begin{align*}
    \frac{\|\tilde{G}\|_F^2}{\alpha - (1-\delta)} &\leq [\beta - (1-\delta)] \cdot \|\tilde{M}-M^*\|_F^2,\\
    \frac{\|\tilde{G}\|_F^2}{(1+\delta) - \alpha} &\leq [(1+\delta) - \beta] \cdot \|\tilde{M}-M^*\|_F^2.
\end{align*}
Summing up the two inequalities and dividing both sides by $2\delta$ gives rise to 
\begin{align}\label{eqn:nes-suf-201} \frac{\|\tilde{G}\|_F^2}{\delta^2 - (1-\alpha)^2}  \leq \|\tilde{M}-M^*\|_F^2. \end{align}
We note that the above condition is also sufficient for the inequalities in \eqref{eqn:rank-1-4-add} to hold by choosing $\beta = 2-\alpha$. Using the relation $\|G\|_F^2 = \|\tilde{G}\|_F^2 + \alpha^2 \|\tilde{M} - M^*\|_F^2$, one can write 
\begin{align}\label{eqn:rank-1-5} \tr(\Lambda^2) = \|G\|_F^2 \leq ( 2\alpha - 1 + \delta^2 ) \|\tilde{M}-M^*\|_F^2 = \alpha^{-1}( 2\alpha - 1 + \delta^2 )\langle \Lambda, R_{r+1:2r, r+1:2r} \rangle. \end{align}
Now, using the fact that $\rank(M^*) \leq r$, we can write the matrix $R$ as
\[ R = \begin{bmatrix} A \\ C \end{bmatrix}\begin{bmatrix} B \\ D \end{bmatrix}^T = \begin{bmatrix} AB^T & AD^T \\ CB^T & CD^T \end{bmatrix}, \]
where $A,B,C,D\in\mathbb{R}^{r\times r}$. Then, conditions \eqref{eqn:rank-1-3} and \eqref{eqn:rank-1-5} become
\begin{align}
\label{eqn:rank-1-6}
    \langle \Lambda, CD^T \rangle &= \alpha \left[ \tr(\Sigma^2) - 2\langle \Sigma, AB^T \rangle + \|AB^T\|_F^2 + \|AD^T\|_F^2 + \|CB^T\|_F^2 + \|CD^T\|_F^2 \right]
\end{align}
and
\begin{align}
\label{eqn:rank-1-7}
    \tr(\Lambda^2) &\leq \alpha^{-1}( 2\alpha - 1 + \delta^2 ) \cdot \langle \Lambda, CD^T \rangle.
\end{align}
If $\langle \Lambda, CD^T \rangle = 0$, we have
\[ \tr(\Sigma^2) - 2\langle \Sigma, AB^T \rangle + \|AB^T\|_F^2 + \|AD^T\|_F^2 + \|CB^T\|_F^2 + \|CD^T\|_F^2 = 0, \]
which implies that
\[ AB^T = \Sigma,\quad AD^T = CB^T = CD^T = 0. \]
This contradicts the assumption that $\tilde{M}\neq M^*$. Combining this with conditions \eqref{eqn:rank-1-1}, \eqref{eqn:rank-1-6} and \eqref{eqn:rank-1-7}, we arrive at the necessity part. For problem \eqref{eqn:asym-reg}, Lemma 3 in \cite{ha2020equivalence} ensures that $\tilde{M}$ is still a fixed point of the SVP algorithm. Recalling the necessary conditions in Theorem \ref{thm:asym-condition}, we know that the same necessary conditions also hold in this case.

\paragraph{Sufficiency.} Now, we study the sufficiency part. We first consider problem \eqref{eqn:asym}. We choose two orthogonal matrices $U\in\mathbb{R}^{n\times 2r},V\in\mathbb{R}^{m\times 2r}$ and define
\[ \tilde{M} =  U \begin{bmatrix} \Sigma & 0_{r\times r}\\ 0_{r\times r} & 0_{r\times r}\\  \end{bmatrix} V^T,\quad M^* := U\left(\begin{bmatrix} A \\ C \end{bmatrix}\begin{bmatrix} B \\ D \end{bmatrix}^T\right) V^T, \quad G := U \begin{bmatrix} 0_{r\times r} & 0_{r\times r}\\ 0_{r\times r} & -\Lambda\\  \end{bmatrix} V^T. \]
Since $\langle \Lambda, CD^T \rangle \neq0$, we have $\tilde{M}\neq M^*$. Then, we know that $\rank(\tilde{M}) \leq r$ and $\rank(M^*) \leq r$. We define
\[ \tilde{G} := G - \alpha(\tilde{M} - M^*), \]
which satisfies $\langle \tilde{G},\tilde{M} - M^*\rangle = 0$ by the condition in the second line of \eqref{eqn:rank-1-thm}. If $\tilde{G} = 0$, then
\[ \begin{bmatrix} 0_{r\times r} & 0_{r\times r}\\ 0_{r\times r} & -\Lambda\\  \end{bmatrix} = \alpha \cdot \begin{bmatrix} \Sigma & 0_{r\times r}\\ 0_{r\times r} & 0_{r\times r}\\  \end{bmatrix} - \alpha \cdot \begin{bmatrix} A \\ C \end{bmatrix}\begin{bmatrix} B \\ D \end{bmatrix}^T = \alpha \cdot \begin{bmatrix} \Sigma & 0_{r\times r}\\ 0_{r\times r} & 0_{r\times r}\\  \end{bmatrix} - \alpha \cdot \begin{bmatrix} AB^T &0 \\ 0& CD^T \end{bmatrix}, \]
where the second step is because of $CB^T=0$ and $AD^T=0$. The above relation is equivalent to
\[ \Sigma = AB^T,\quad \Lambda = \alpha \cdot CD^T. \]
Since $\Sigma \succ 0$, the matrix $AB^T$ has rank $r$. Noticing that the decomposition of matrix $M^*$ ensures that the rank of $M^*$ is at most $r$, we have $CD^T = 0$, which is a contradiction to the condition that $\langle CD^T,\Lambda\rangle \neq 0$. Therefore, we have $\tilde{G} \neq 0$. We consider the rank-$2$ symmetric tensor
\begin{align*} 
\mathcal{G}_1 :=& \frac{\alpha}{\|\tilde{M}-M^*\|_F^2} \cdot (\tilde{M}-M^*)\otimes (\tilde{M}-M^*) + \frac{2-\alpha}{\|\tilde{G}\|_F^2} \cdot \tilde{G}\otimes \tilde{G} \\
&+ \frac{1}{\|\tilde{M}-M^*\|_F^2}\left[ (\tilde{M}-M^*)\otimes \tilde{G} + \tilde{G} \otimes (\tilde{M}-M^*) \right] .
\end{align*}
For every matrix $K \in \mathbb{R}^{n\times m}$, we have the decomposition
\[ K = t(\tilde{M} - M^*) + s \tilde{G} + \tilde{K},\quad \langle \tilde{M}-M^*, \tilde{K}\rangle = \langle \tilde{G},\tilde{K}\rangle = 0,  \]
where $t,s\in\mathbb{R}$ are two suitable constants. Then, using the definition of $\mathcal{G}_1$, we have
\[ K : \mathcal{G}_1 : K = \alpha \|\tilde{M}-M^*\|_F^2 \cdot t^2 + 2 \|\tilde{G}\|_F^2 \cdot ts + (2-\alpha) \|\tilde{G}\|_F^2 \cdot s^2. \]
By the conditions in the third line of \eqref{eqn:rank-1-thm}, one can write 
\[ \|\tilde{G}\|_F^2 \leq [\alpha - (1-\delta)][(1+\delta) - \alpha] \cdot \|\tilde{M}-M^*\|_F^2, \]
which leads to
\begin{align*}
\nonumber&[\alpha - (1-\delta)] \|\tilde{M}-M^*\|_F^2 \cdot t^2 + 2 \|\tilde{G}\|_F^2 \cdot ts+ [(1+\delta) - \alpha] \|\tilde{G}\|_F^2 \cdot s^2 \geq0,\\
&[(1+\delta) - \alpha] \|\tilde{M}-M^*\|_F^2 \cdot t^2 - 2 \|\tilde{G}\|_F^2 \cdot ts + [\alpha - (1-\delta)] \|\tilde{G}\|_F^2 \cdot s^2 \geq0.
\end{align*}
The above two inequalities are equivalent to
\begin{align}\label{eqn:rank-1-4-suf} (1-\delta) [\|\tilde{M}-M^*\|_F^2 \cdot s^2+\|\tilde{G}\|_F^2 \cdot t^2] \leq  K : \mathcal{G}_1 : K \leq (1+\delta) [\|\tilde{M}-M^*\|_F^2 \cdot s^2+\|\tilde{G}\|_F^2 \cdot t^2] . \end{align}
By restricting to the subspace
\[ \mathcal{S} := \mathrm{span}\{\tilde{M}-M^*,\tilde{G}\} = \{ s(\tilde{M}-M^*) + t\tilde{G}~|~ s,t\in\mathbb{R} \}, \]
the tensor $\mathcal{G}_1$ can be viewed as a $2\times 2$ matrix. Then, inequality \eqref{eqn:rank-1-4-suf} implies that the matrix has two eigenvalues $\lambda_1$ and $\lambda_2$ such that 
\[ 1-\delta \leq \lambda_1, \lambda_2 \leq 1+\delta. \]
Therefore, we can rewrite the tensor $\mathcal{G}_1$ restricted to $\mathcal{S}$ as
%
\begin{align*}
    [\mathcal{G}_1]_\mathcal{S} = \lambda_1 \cdot G_1 \otimes G_1 + \lambda_2 \cdot G_2 \otimes G_2,
\end{align*}
where $G_1,G_2$ are linear combinations of $\tilde{M}-M^*,\tilde{G}$ and have the unit norm. Since the orthogonal complementary $\mathcal{S}^\perp$ is in the null space of $\mathcal{G}_1$, we have
\[ \mathcal{G}_1 = [\mathcal{G}_1]_\mathcal{S} = \lambda_1 \cdot G_1 \otimes G_1 + \lambda_2 \cdot G_2 \otimes G_2. \]
Now, we choose matrices $G_3,\dots,G_{N}$ such that $G_1,\dots,G_N$ form an orthonormal basis of the linear vector space $\mathbb{R}^{n\times m}$, where $N := nm$. We define another symmetric tensor by
\begin{align*}
    \mathcal{H} := \mathcal{G}_1 + \sum_{i=3}^N (1+\delta) \cdot G_i \otimes G_i.
\end{align*}
Then, inequality \eqref{eqn:rank-1-4-suf} implies that the quadratic form $K: \mathcal{H} : K$ satisfies the $\delta$-RIP$_{2r,2r}$ property.

Therefore, we can choose the Hessian to be the constant tensor $\mathcal{H}$ and define the function $f_a(\cdot)$ as
\begin{align*}
    f_a(K) := \frac{1}{2} (K-M^*) : \mathcal{H} : (K-M^*), \quad \forall K\in\mathbb{R}^{n\times m}.
\end{align*}
Combining with the definition of $\mathcal{H}$, we know
\[ \nabla f_a(\tilde{M}) = \mathcal{H} : (\tilde{M}-M^*) = G,\quad \nabla^2 f_a(\tilde{M}) = \mathcal{H}. \]
We choose matrices $\bar{U}\in\mathbb{R}^{n\times r},\bar{V}\in\mathbb{R}^{m\times r}$ such that $\tilde{M} = \bar{U}\bar{V}^T$ and $\bar{U}^T\bar{U}=\bar{V}^T\bar{V}$. By the definitions of $\tilde{M}$ and $G$, we know that $\tilde{M}$ and $G$ have orthogonal column and row spaces, i.e.,
\[ \bar{U}^TG = 0,\quad G\bar{V} = 0. \]
This means that the first-order optimality conditions are satisfied at the point $(\bar{U},\bar{V})$. For the second-order necessary optimality conditions, we consider the direction
\[ \Delta := \begin{bmatrix} \Delta_U \\ \Delta_V \end{bmatrix} \in \mathbb{R}^{(n+m)\times r}. \]
We consider the decomposition
\[ \Delta_U = \mathcal{P}_{\bar{U}}\Delta_U + \mathcal{P}_{\bar{U}}^\perp\Delta_U := \Delta_U^1 + \Delta_U^2,\quad \Delta_V = \mathcal{P}_{\bar{V}}\Delta_V + \mathcal{P}_{\bar{V}}^\perp\Delta_V := \Delta_V^1 + \Delta_V^2, \]
where $\mathcal{P}_{\bar{U}},\mathcal{P}_{\bar{V}}$ are the orthogonal projection onto the column space of $\bar{U},\bar{V}$, respectively. Then, using the conditions in the first line of \eqref{eqn:rank-1-thm}, we have
\begin{align}
\label{eqn:rank-suf-2}
\nonumber\langle \nabla f_a(\tilde{M}), \Delta_U\Delta_V^T\rangle &= \langle G, \Delta_U\Delta_V^T \rangle = \langle G, \Delta_U^2(\Delta_V^2)^T \rangle \geq - \|G^T\Delta_U^2\|_F\|\Delta_V^2\|_F\\
&\geq -(1+\delta)\sigma_r(\tilde{M})\|\Delta_U^2\|_F\|\Delta_V^2\|_F \geq -(1+\delta)\sigma_r(\tilde{M}) \cdot \frac{\|\Delta_U^2\|_F^2 + \|\Delta_V^2\|_F^2}{2}.
\end{align}
We define
\[ \Delta_1 := \bar{U}(\Delta_V^1)^T + \Delta_U^1\bar{V}^T,\quad \Delta_2 := \bar{U}(\Delta_V^2)^T + \Delta_U^2\bar{V}^T. \]
Then, we know that $\langle \Delta_1,\Delta_2\rangle = 0$. Using the assumption that $CB^T=AD^T = 0$, we know that $M^*$ has the form
\begin{align} \label{eqn:nes-suf-101}
M^* = U \begin{bmatrix} AB^T &0 \\ 0& CD^T \end{bmatrix} V^T = \mathcal{P}_{\bar{U}} M^* \mathcal{P}_{\bar{V}} + \mathcal{P}_{\bar{U}}^\perp M^* \mathcal{P}_{\bar{V}}^\perp.
\end{align}
%
Then, the special form \eqref{eqn:nes-suf-101} implies that
\[ \langle M^*,\Delta_2 \rangle = \langle M^* , \bar{U}(\Delta_V^2)^T + \Delta_U^2\bar{V}^T\rangle = \left\langle M^* , \bar{U}\Delta_V^T\mathcal{P}_{\bar{V}}^\perp + \mathcal{P}_{\bar{U}}^\perp\Delta_U\bar{V}^T \right\rangle = 0. \]
%
Using the definitions of $\tilde{M}$ and $G$, it can be concluded that 
\[ \langle \tilde{M} , \Delta_2\rangle = 0,\quad \langle G,\Delta_2 \rangle = \langle G , \bar{U}(\Delta_V^2)^T + \Delta_U^2\bar{V}^T\rangle = 0. \]
Since $G_1,G_2$ are linear combinations of $\tilde{M}-M^*$ and $G$, the last three relations lead to
\[ \langle G_1 , \Delta_2\rangle = \langle G_2, \Delta_2\rangle = 0. \]
Therefore, there exist constants $a_3,\dots,a_N$ such that
\[ \Delta_2 = \sum_{i=3}^N a_i G_i. \]
Suppose that the constants $b_1,\dots,b_N$ satisfy
\[ \Delta_1 = \sum_{i=1}^N b_i G_i. \]
Then, the fact $\langle\Delta_1,\Delta_2\rangle = 0$ and the orthogonality of $G_1,\dots,G_N$ imply that
\[ \sum_{i=3}^N a_ib_i = 0. \]
We can calculate that
\begin{align*}
    \nonumber &[\nabla^2 f_a(\tilde{M})]( \bar{U}\Delta_V^T + \Delta_U \bar{V}^T,\bar{U}\Delta_V^T + \Delta_U \bar{V}^T ) = (\Delta_1 + \Delta_2) : \mathcal{H} : (\Delta_1 + \Delta_2)\\
    =& \lambda_1 \cdot b_1^2 + \lambda_2 \cdot b_2^2 + (1+\delta)  \sum_{i=3}^N (a_i + b_i)^2 \geq (1+\delta)  \sum_{i=3}^N (a_i + b_i)^2\\
    =& (1+\delta)  \sum_{i=3}^N \left(a_i^2 + b_i^2\right) \geq (1+\delta)  \sum_{i=3}^N a_i^2 = (1+\delta) \|\bar{U}(\Delta_V^2)^T + \Delta_U^2\bar{V}^T\|_F^2,
\end{align*}
where the third last step is due to $\sum_{i=3}^N a_ib_i = 0$.
%
%
Noticing that $\langle \bar{U}(\Delta_V^2)^T, \Delta_U^2\bar{V}^T\rangle = 0 $, the above inequality gives that
\begin{align*}
    &[\nabla^2 f_a(\tilde{M})]( \bar{U}\Delta_V^T + \Delta_U \bar{V}^T,\bar{U}\Delta_V^T + \Delta_U \bar{V}^T ) \geq (1+\delta) \|\bar{U}(\Delta_V^2)^T\|_F^2 + (1+\delta) \|\Delta_U^2\bar{V}^T\|_F^2\\
    \geq & (1+\delta)\sigma_r(\bar{U})^2\|\Delta_V^2\|_F^2 + (1+\delta)\sigma_r(\bar{V})^2\|\Delta_U^2\|_F^2 = (1+\delta)\sigma_r(\tilde{M}) ( \|\Delta_V^2\|_F^2 + \|\Delta_U^2\|_F^2 ),
\end{align*}
%
where the last equality is because of $\sigma_r(\bar{U})^2 = \sigma_r(\bar{V})^2 = \sigma(\tilde{M})$ when $\bar{U}^T\bar{U}=\bar{V}^T\bar{V}$. Combining with inequality \eqref{eqn:rank-suf-2}, one can write 
\begin{align*}
    &[\nabla^2 h_a(U,V)](\Delta,\Delta) = 2\langle \nabla f_a(\tilde{M}), \Delta_U\Delta_V^T\rangle + [\nabla^2 f_a(\tilde{M})]( \bar{U}\Delta_V^T + \Delta_U \bar{V}^T,\bar{U}\Delta_V^T + \Delta_U \bar{V}^T )\\
    \geq& -(1+\delta)\sigma_r(\tilde{M})( \|\Delta_V^2\|_F^2 + \|\Delta_U^2\|_F^2 ) + (1+\delta)\sigma_r(\tilde{M})( \|\Delta_V^2\|_F^2 + \|\Delta_U^2\|_F^2 ) = 0.
\end{align*}
This shows that $(\bar{U},\bar{V})$ satisfies the second-order necessary optimality conditions, and therefore it is a spurious second-order critical point. 

Now, we consider problem \eqref{eqn:asym-reg}. Since the point $(\bar{U},\bar{V})$ satisfies $\bar{U}^T\bar{U}=\bar{V}^T\bar{V}$, it is also a local minimum of the regularization term. Hence, the point $(\bar{U},\bar{V})$ is also a spurious second-order critical point of the regularized problem \eqref{eqn:asym-reg}.
\end{proof}

\subsection{Proof of Corollary \ref{cor:rank-1}}

\begin{proof}[Proof of Corollary \ref{cor:rank-1}]
We assume that problem \eqref{eqn:asym} has a spurious second-order critical point. By the necessity part of Theorem \eqref{eqn:rank-1-thm}, there exist $\alpha \in (1-\delta,1+\delta)$ and real numbers $\sigma,\lambda, a,b,c,d$ such that
\begin{align}\label{eqn:rank-2-1}
    \nonumber(1+\delta)\sigma &\geq \lambda > 0,\quad \alpha^{-1}(2\alpha - 1 + \delta^2) cd\cdot \lambda \geq \lambda^2 > 0,\\
    cd \cdot \lambda &= \alpha[ \sigma^2 - 2 ab\cdot \sigma + (ab)^2 + (ad)^2 + (cb)^2 + (cd)^2 ].
\end{align}
We first relax the second line to
\begin{align}\label{eqn:rank-2-2}
    cd \cdot \lambda &\geq \alpha[ \sigma^2 - 2 |ab|\cdot \sigma + (ab)^2 + 2|ab| \cdot |cd| + (cd)^2 ].
\end{align}
Then, we denote $x:=|ab|$ and consider the quadratic programming problem
\[ \min_{x\geq0}~ x^2 + 2( |cd| - \sigma ) \cdot x, \]
whose optimal value is
\[ - ( \sigma - |cd| )_+^2, \]
where $(t)_+ := \max\{t,0\}$. Substituting into inequality \eqref{eqn:rank-2-2}, we obtain
\begin{align}\label{eqn:rank-2-3}
    cd \cdot \lambda &\geq \alpha[ \sigma^2 - ( \sigma - |cd| )_+^2 + (cd)^2 ].
\end{align}
Then, we consider two different cases.

\paragraph{Case I.} We first consider the case when $\sigma \geq |cd|$. In this case, the inequality \eqref{eqn:rank-2-3} becomes
\begin{align*}
    cd \cdot \lambda &\geq 2\alpha \cdot \sigma |cd| = 2\alpha \cdot \sigma cd,
\end{align*}
where the last equality is due to $cd > 0$. Therefore,
\[ \lambda \geq 2\alpha \cdot \sigma. \]
The second inequality in \eqref{eqn:rank-2-1} implies $\lambda \leq \alpha^{-1}(2\alpha - 1 + \delta^2) \cdot cd$. Combining with the above inequality and the assumption of this case, it follows that
\begin{align*}
     \alpha^{-1}(2\alpha - 1 + \delta^2) \cdot \sigma \geq \alpha^{-1}(2\alpha - 1 + \delta^2) \cdot cd \geq 2\alpha \cdot \sigma,
\end{align*}
which is further equivalent to 
\[ \alpha^{-1}(2\alpha - 1 + \delta^2) \geq 2\alpha \iff \delta^2 \geq 2\alpha^2 - 2\alpha + 1. \]
Since $2\alpha^2 - 2\alpha + 1 \geq 1/2$, we arrive at $\delta^2 \geq 1/2$, which is a contradiction to $\delta < 1/2$.

\paragraph{Case II.} We then consider the case when $\sigma \leq |cd|$. In this case, the inequality \eqref{eqn:rank-2-3} becomes
\begin{align*}
    cd \cdot \lambda &\geq \alpha[ \sigma^2 + (cd)^2 ].
\end{align*}
Combining with the second inequality in \eqref{eqn:rank-2-1}, we obtain $\lambda \leq \alpha^{-1}(2\alpha - 1 + \delta^2) \cdot (cd)$. Therefore,
\begin{align*}
    \alpha^{-1}(2\alpha - 1 + \delta^2) \cdot (cd)^2  \geq cd \cdot \lambda\geq \alpha[ \sigma^2 + (cd)^2 ].
\end{align*}
Moreover, the first inequality in \eqref{eqn:rank-2-1} gives
\begin{align*}
    (1+\delta) \sigma \cdot cd  \geq cd \cdot \lambda\geq \alpha[ \sigma^2 + (cd)^2 ].
\end{align*}
By denoting $y := cd$, the above two inequalities become
\begin{align}\label{eqn:rank-2-6}
    \nonumber\alpha^{-1}(2\alpha - 1 + \delta^2) \cdot y^2  &\geq \alpha[ \sigma^2 + y^2 ],\\
    (1+\delta) \sigma \cdot y  &\geq \alpha[ \sigma^2 + y^2 ].
\end{align}
By denoting $z := y / \sigma$, the first inequality in \eqref{eqn:rank-2-6} implies
\begin{align}\label{eqn:rank-2-7}
    z^2 \geq \frac{\alpha^2}{\delta^2 - (1-\alpha)^2}.
\end{align}
Since $\delta < 1/2$, one can write 
\[ (1-\alpha)^2 + \alpha^2 \geq \frac12 > \frac14 > \delta^2, \]
which is equivalent to $\alpha^2 \geq \delta^2 - (1-\alpha)^2$. Therefore, inequality \eqref{eqn:rank-2-7} implies that $z^2 \geq 1$ and
\begin{align}
\label{eqn:rank-2-8}
    z^2 + \frac{1}{z^{2}} \geq \frac{\alpha^2}{\delta^2 - (1-\alpha)^2} + \frac{\delta^2 - (1-\alpha)^2}{\alpha^2}.
\end{align}
On the other hand, the second inequality in \eqref{eqn:rank-2-6} implies
\begin{align*}
    z + \frac{1}{z} \leq \frac{1+\delta}{\alpha} \quad \text{and thus}\quad z^2 + \frac{1}{z^2} + 2 \leq \frac{(1+\delta)^2}{\alpha^2}.
\end{align*}
Combining with inequality \eqref{eqn:rank-2-8}, it follows that
\begin{align}\label{eqn:rank-r-add}
    \frac{\alpha^2}{\delta^2 - (1-\alpha)^2} + \frac{\delta^2 - (1-\alpha)^2}{\alpha^2} + 2 \leq \frac{(1+\delta)^2}{\alpha^2}.
\end{align}
By some calculation, the above inequality is equivalent to
\[ (\delta^2+2\delta+5) \cdot \alpha^2 + (2\delta^2 - 4\delta - 6) \cdot \alpha + 2(1+\delta)(1-\delta^2) \leq 0. \]
Checking the discriminant of the above quadratic function, we obtain
\[ (2\delta^2 - 4\delta - 6)^2 - 8(\delta^2+2\delta+5)(1+\delta)(1-\delta^2) \geq0, \]
which is equivalent to
\[ 4(2\delta - 1)(\delta+1)^4 \geq 0. \]
However, the above claim contradicts the assumption that $\delta < 1/2$.

In summary, the contradictions in the two cases imply that the condition \eqref{eqn:rank-2-1} cannot hold, and therefore there does not exist spurious second-order critical points.
\end{proof}

\subsection{Counterexample for the Rank-one Case}

\begin{example}\label{exm:1}
Let $e_i\in\mathbb{R}^n$ be the $i$-th standard basis of $\mathbb{R}^n$. We define the tensor
\begin{align*}
    \mathcal{H} := &\sum_{i,j=1}^n (e_i e_j^T) \otimes (e_ie_j^T) + \frac12(e_1 e_1^T) \otimes (e_2 e_2^T) + \frac12(e_2 e_2^T) \otimes (e_1 e_1^T)\\
    &+ \frac14 \left[(e_1 e_2^T) \otimes (e_1 e_2^T) + (e_2 e_1^T) \otimes (e_2 e_1^T)\right] + \frac14 (e_1 e_2^T) \otimes (e_2 e_1^T)+ \frac14 (e_2 e_1^T) \otimes (e_1 e_2^T)
\end{align*}
and the objective function
\[ f_a(M) := (M - e_1e_1^T) : \mathcal{H} : (M - e_1e_1^T) \quad \forall M \in \mathbb{R}^{n\times n}. \]
The global minimizer of $f_a(\cdot)$ is the rank-$1$ matrix $M^*:=e_1e_1^T$. It has been proved in \cite{zhang2019sharp} that the function $f_a(\cdot)$ satisfies the $\delta$-RIP$_{2,2}$ property with $\delta = 1/2$. Moreover, we define
\[ U:= \frac{1}{\sqrt{2}} e_2,\quad \tilde{M} := UU^T \neq M^*. \]
It has been proved in \cite{zhang2019sharp} that the first-order optimality condition is satisfied. To verify the second-order necessary condition, we can calculate that
\begin{align*}
    [\nabla^2 h_a(U,U)](\Delta,\Delta) &= 2\langle \nabla f_a(\tilde{M}), \Delta_U \Delta_V^T\rangle + (U\Delta_V^T + \Delta_U U^T) : \mathcal{H} : (U\Delta_V^T + \Delta_U U^T)\\
    &= -\frac32 (\Delta_U)_1(\Delta_V)_1 + \frac58 \left[ (\Delta_U)_1^2 + (\Delta_V)_1^2 \right] + \frac14(\Delta_U)_1(\Delta_V)_1\\
    &\quad+ \frac12\left[ (\Delta_U)_2 + (\Delta_V)_2 \right]^2  + \frac12 \sum_{i=3}^n \left[ (\Delta_U)_i^2 +(\Delta_V)_i^2 \right]\\
    &= \frac58\left[ (\Delta_U)_1 - (\Delta_V)_1 \right]^2 + \frac12\left[ (\Delta_U)_2 + (\Delta_V)_2 \right]^2  + \frac12 \sum_{i=3}^n \left[ (\Delta_U)_i^2 +(\Delta_V)_i^2 \right],
\end{align*}
which is non-negative for every $\Delta\in\mathbb{R}^n$. Hence, we conclude that the point $\tilde{M}$ is a spurious second-order critical point of problem \eqref{eqn:asym}. Moreover, since we choose $V = U$, the point $\tilde{M}$ is a global minimizer of the regularizer $\|U^TU-V^TV\|_F^2$ and thus $\tilde{M}$ is also a spurious second-order critical point of problem \eqref{eqn:asym-reg}.
\end{example}

\subsection{Proof of Corollary \ref{cor:rank-r}}

\begin{proof}[Proof of Corollary \ref{cor:rank-r}]
We first consider the case when $\delta \leq 1/3$. We assume that there exists a spurious second-order critical point $\tilde{M}$. Then, by Theorem \ref{thm:nes-suf}, we know that there exists a constant $\alpha \in(1-\delta, (1+\delta)/2]$. This means that
\[ 1 - \delta < \frac{1+\delta}{2}, \]
which contradicts the assumption that $\delta \leq 1/3$.

Then, we consider the case when $\delta < 1/2$. With no loss of generality, assume that $\tilde{M}\neq M^*$ and $M^* \neq 0$; otherwise, the inequality in this theorem is trivially true. Define
\[ m_{11} := \|\Sigma\|_F^2,\quad m_{12} := \langle \Sigma, AB^T\rangle ,\quad m_{22} :=\|AB^T\|_F^2 + \|AD^T\|_F^2 + \|CB^T\|_F^2 + \|CD^T\|_F^2. \]
By our construction in Theorem \ref{thm:nes-suf}, we know that
\[ m_{11} = \|\tilde{M}\|_F^2,\quad m_{12} = \langle \tilde{M},M^* \rangle,\quad m_{22} = \|M^*\|_F^2. \]
Therefore, we only need to prove $m_{12} \geq C(\delta)\cdot \sqrt{m_{11}m_{22}}$ for some constant $C(\delta) > 0$. By the analysis in \cite{ha2020equivalence}, we know that the second-order critical point $\tilde{M}$ must have rank $r$ and thus $m_{11}\neq 0$. The remainder of the proof is split into two steps.

\paragraph{Step I.} First, we prove that
\begin{align}\label{eqn:rank-r-1}
    \frac{(m_{11} + m_{22} - 2m_{12})^2}{m_{11}m_{22} - m_{12}^2} \leq \frac{(1+\delta)^2}{\alpha^2},\quad \frac{(m_{11} - m_{12})^2}{m_{11}m_{22} - m_{12}^2} \leq  \frac{\delta^2 - (1-\alpha)^2}{\alpha^2}.
\end{align}
We first rule out the case when $m_{11}m_{22} - m_{12}^2 = 0$. In this case, the equality condition of the Cauchy inequality shows that there exists a constant $t$ such that
\[ \tilde{M} = tM^*. \]
Since $\tilde{M}\neq 0$, the constant $t$ is not $0$. Using the mean value theorem, there exists a constant $c\in[0,1]$ such that
\begin{align*} 
\langle \nabla f_a(\tilde{M}), Z\rangle &= \nabla^2 f[M^* + c(\tilde{M} - M^*)](\tilde{M} - M^*,Z)\\
&= \nabla^2 f[M^* + c(\tilde{M} - M^*)][ (t - 1) M^*, Z], \quad \forall Z\in\mathbb{R}^{n\times m}.
\end{align*}
The $\delta$-RIP$_{2r,2r}$ property gives
\[ \langle \nabla f_a(\tilde{M}), \tilde{M} \rangle = \nabla^2 f[M^* + c(\tilde{M} - M^*)][ (t-1)M^*, tM^* ] \geq t(t-1)(1-\delta)\|M^*\|_F^2. \]
If $t = 1$, we conclude that $\tilde{M} = M^*$, which contradicts the assumption that $\tilde{M}\neq M^*$. Therefore, it holds that
\[ \langle \tilde{M}, \nabla f_a(\tilde{M}) \rangle \neq 0. \]
This contradicts the first-order optimality condition, which states that $\langle \tilde{M},\nabla f_a(\tilde{M}) \rangle = 0$. Hence, we have proved that inequality \eqref{eqn:rank-r-1} is well defined. We consider the decomposition
\begin{align*}
    \begin{bmatrix}0&0\\0&\Lambda\end{bmatrix} = c_1 \begin{bmatrix}\Sigma&0\\0&0\end{bmatrix} + c_2 \begin{bmatrix} A \\ C\end{bmatrix}\begin{bmatrix} B\\D\end{bmatrix}^T + K,\quad \left\langle K, \begin{bmatrix}\Sigma&0\\0&0\end{bmatrix} \right\rangle = \left\langle K, \begin{bmatrix} A \\ C\end{bmatrix}\begin{bmatrix} B\\D\end{bmatrix}^T \right\rangle = 0.
\end{align*}
%
Using the conditions in Theorem \ref{thm:nes-suf}, it follows that
\begin{align*}
    \left\langle \begin{bmatrix}0&0\\0&\Lambda\end{bmatrix}, \begin{bmatrix}\Sigma&0\\0&0\end{bmatrix} \right\rangle = 0,\quad \left\langle \begin{bmatrix}0&0\\0&\Lambda\end{bmatrix}, \begin{bmatrix} A \\ C\end{bmatrix}\begin{bmatrix} B\\D\end{bmatrix}^T \right\rangle = \alpha( m_{11} - 2m_{12} + m_{22} ).
\end{align*}
The pair of coefficients $(c_1,c_2)$ can be uniquely solved as
\begin{align*}
    c_1 = -\alpha \cdot \frac{m_{11}+m_{22}-2m_{12}}{m_{11}m_{22}-m_{12}^2} \cdot m_{12},\quad c_2 =  \alpha \cdot \frac{m_{11}+m_{22}-2m_{12}}{m_{11}m_{22}-m_{12}^2} \cdot m_{11}.
\end{align*}
Using the orthogonality of the decomposition, we have
\begin{align}\label{eqn:rank-r-4}
    \nonumber\|\Lambda\|_F^2 &\geq \left\| c_1 \begin{bmatrix}\Sigma&0\\0&0\end{bmatrix} + c_2 \begin{bmatrix} A \\ C\end{bmatrix}\begin{bmatrix} B\\D\end{bmatrix}^T \right\|_F^2 =  c_1^2 m_{11} + 2c_1c_2 m_{12} + c_2^2 m_{22}\\
    &= \alpha^2 \cdot \frac{m_{11}(m_{11} + m_{22} - 2m_{12})^2}{m_{11}m_{22} - m_{12}^2}.
\end{align}
Using the last two lines of condition \eqref{eqn:rank-1-thm}, one can write 
\begin{align*}
    &\alpha^2 \cdot \frac{m_{11}(m_{11} + m_{22} - 2m_{12})^2}{m_{11}m_{22} - m_{12}^2} \leq \|\Lambda\|_F^2\\
    &\leq (2\alpha - 1 + \delta^2) \left[ \tr(\Sigma^2) - 2\langle \Sigma, AB^T \rangle + \|AB^T\|_F^2 + \|AD^T\|_F^2 + \|CB^T\|_F^2 + \|CD^T\|_F^2 \right]\\
    &= (2\alpha - 1 + \delta^2)(m_{11} - 2m_{12} + m_{22}).
\end{align*}
Simplifying the above inequality, we arrive at the second inequality in \eqref{eqn:rank-r-1}. Now, the first inequality in condition \eqref{eqn:rank-1-thm} implies that
\[ \|\Lambda\|_F^2 \leq (1+\delta)^2 \|\Sigma\|_F^2 = (1+\delta)^2 m_{11}. \]
Substituting inequality \eqref{eqn:rank-r-4} into the left-hand side, it follows that
\begin{align*}
    \alpha^2 \cdot \frac{m_{11}(m_{11} + m_{22} - 2m_{12})^2}{m_{11}m_{22} - m_{12}^2} \leq (1+\delta)^2 m_{11},
\end{align*}
which is equivalent to the first inequality in \eqref{eqn:rank-r-1}.

\paragraph{Step II.} Next, we prove the existence of $C(\delta)$. We denote
\[ \kappa := \frac{m_{12}}{\sqrt{m_{11}m_{22}}} \in (-1,1). \]
%
and
\[ C_1 := \frac{\delta^2 - (1-\alpha)^2}{\alpha^2},\quad C_2 := \frac{(1+\delta)^2}{\alpha^2},\quad t := \sqrt{\frac{m_{11}}{m_{22}}}. \]
Since $\tilde{M}\neq 0$, we have $t> 0$. The inequalities in \eqref{eqn:rank-r-1} can be written as
\begin{align}\label{eqn:rank-r-2}
    ( t - \kappa )^2 \leq (1 - \kappa ^2)  C_1,\quad ( t + 1/t - 2\kappa )^2 \leq (1 - \kappa ^2) C_2.
\end{align}
Using the assumption that $\delta < 1/2$, we can write 
\[ \delta^2 < \frac14 < (1-\alpha)^2 + \frac12\alpha^2, \]
which leads to
\begin{align*}
    C_1 = \frac{\delta^2 - (1-\alpha)^2}{\alpha^2} < \frac12.
\end{align*}
If $\kappa + \sqrt{(1-\kappa^2)C_1} \geq 1$, then
\begin{align}\label{eqn:rank-r-3}
    |\kappa| \geq \frac{1-C_1}{1+C_1} \geq \frac13 > 0.
\end{align}
If $\kappa < 0$, then it holds that
\[ \kappa + \sqrt{(1-\kappa^2)C_1} \leq -\frac13 + \sqrt{\frac12} < 1, \]
which contradicts the assumption. Therefore, we have $\kappa \geq0$ and inequality \eqref{eqn:rank-r-3} gives $\kappa \geq 1/3$.

Now, we assume that $\kappa + \sqrt{(1-\kappa^2)C_1} \leq 1$. Then, the first inequality in \eqref{eqn:rank-r-2} gives
\[ 0 < t \leq \kappa + \sqrt{(1-\kappa^2)C_1} \leq 1, \]
which further leads to
\[ t + \frac1t - 2\kappa \geq -\kappa + \sqrt{(1-\kappa^2)C_1} + \frac{1}{\kappa + \sqrt{(1-\kappa^2)C_1}}. \]
Combining with the second inequality in \eqref{eqn:rank-r-2}, we obtain
\begin{align*}
    -\kappa + \sqrt{(1-\kappa^2)C_1} + \frac{1}{\kappa + \sqrt{(1-\kappa^2)C_1}} \leq \sqrt{(1 - \kappa ^2) C_2}.
\end{align*}
The above inequality can be simplified to
\begin{align*}
    \sqrt{1-\kappa^2}( 1 + C_1 - \sqrt{C_1C_2} ) \leq \kappa \sqrt{C_2}.
\end{align*}
We notice that the inequality $1 + C_1 - \sqrt{C_1C_2} \leq 0$ is equivalent to inequality \eqref{eqn:rank-r-add}, which cannot hold when $\delta < 1/2$. Therefore, we have $1 + C_1 - \sqrt{C_1C_2} > 0$ and $\kappa > 0$. Then, the above inequality is equivalent to
\[ (1 - \kappa^2)( 1 + C_1 - \sqrt{C_1C_2} )^2 \leq \kappa^2 \cdot C_2. \]
Therefore, we have
\begin{align*}
    \kappa^2 \geq \frac{( 1 + C_1 - \sqrt{C_1C_2} )^2}{ ( 1 + C_1 - \sqrt{C_1C_2} )^2 + C_2 } = 1 - \frac{1}{1 + \eta^2},
\end{align*}
where we define
\[ \eta := \frac{1 + C_1 - \sqrt{C_1C_2}}{\sqrt{C_2}}. \]
To prove the existence of $C(\delta)$ such that $\kappa \geq C(\delta) > 0$, we only need to show that $\eta$ is lower bounded by a positive constant. With $\delta$ fixed, $\eta$ can be viewed as a continuous function of $\alpha$. Since $\eta = (1-\delta)/(1+\delta) > 0$ when $\alpha = 1 - \delta$, the function/parameter $\eta$ is defined for all $\alpha$ in the compact set $[1-\delta,(1+\delta)/2]$. Combining with the fact that $1 + C_1 - \sqrt{C_1C_2} > 0$, the function $\eta$ is positive on a compact set, and thus there exists a positive lower bound $\bar{C}(\delta) > 0$.

In summary, we can define the function
\[ C(\delta) := \min\left\{ \frac13, \bar{C}(\delta) \right\} >0 \]
such that $\kappa \geq C(\delta)$ for every spurious second-order critical point $\tilde{M}$. 
\end{proof}

\subsection{Counterexample for the General Rank Case with Linear Measurements}

\begin{example}

Using the previous rank-$1$ example, we design a counterexample with linear measurement for the rank-$r$ case. Let $n\geq 2r$ be an integer and $e_i\in\mathbb{R}^n$ be the $i$-th standard basis of $\mathbb{R}^n$. We define the tensor
\begin{align*}
    \mathcal{H} := &\frac32\sum_{i,j=1}^n (e_i e_j^T) \otimes (e_ie_j^T) + \sum_{i=1}^r \Big\{-\frac12\left[(e_{2i-1} e_{2i-1}^T) \otimes (e_{2i-1} e_{2i-1}^T) + (e_{2i} e_{2i}^T) \otimes (e_{2i} e_{2i}^T)\right]\\
    &+\frac12\left[(e_{2i-1} e_{2i-1}^T) \otimes (e_{2i} e_{2i}^T) + (e_{2i} e_{2i}^T) \otimes (e_{2i-1} e_{2i-1}^T)\right]\\
    &- \frac14 \left[(e_{2i-1} e_{2i}^T) \otimes (e_{2i-1} e_{2i}^T) + (e_{2i} e_{2i-1}^T) \otimes (e_{2i} e_{2i-1}^T)\right]\\
    &+ \frac14 \left[(e_{2i-1} e_{2i}^T) \otimes (e_{2i} e_{2i-1}^T)+ (e_{2i} e_{2i-1}^T) \otimes (e_{2i-1} e_{2i}^T)\right]\Big\}
\end{align*}
and the rank-$r$ global minimum
\[ U^* := \begin{bmatrix} e_1 & e_3 &\cdots & e_{2r-1} \end{bmatrix},\quad  M^* := U^*(U^*)^T = \sum_{i=1}^r e_{2i-1} e_{2i-1}^T. \]
The objective function is defined as
\[ f_a(M) := (M - M^*) : \mathcal{H} : (M - M^*) \quad \forall M \in \mathbb{R}^{n\times n}. \]
We can similarly prove that the function $f_a(\cdot)$ satisfies the $\delta$-RIP$_{2r,2r}$ property with $\delta = 1/2$. Moreover, we define
\[ \tilde{U} := \frac{1}{\sqrt{2}}\begin{bmatrix} e_2 & e_4 &\cdots & e_{2r} \end{bmatrix},\quad \tilde{M}:= \tilde{U}\tilde{U}^T = \frac{1}{2} \sum_{i=1}^r e_{2i}e_{2i}^T  \neq M^*. \]
The gradient of $f_a(\cdot)$ at point $\tilde{M}$ is
\[ \nabla f_a(\tilde{M}) = - \frac{3}{4} \sum_{i=1}^r e_{2i-1} e_{2i-1}^T \in \mathbb{R}^{2r\times 2r}. \]
Since the column and row spaces of the gradient are orthogonal to those of $\tilde{M}$, the first-order optimality condition is satisfied. To verify the second-order necessary condition, we can similarly calculate that
\begin{align*}
    [\nabla^2 h_a(\tilde{U},\tilde{U})]&(\Delta,\Delta)\\
    =& 2 \langle \nabla f_a(\tilde{M}), \Delta_U \Delta_V^T\rangle + (\tilde{U}\Delta_V^T + \Delta_V \tilde{U}^T) : \mathcal{H} : (\tilde{U}\Delta_V^T + \Delta_U \tilde{U}^T)\\
    =& -\frac32 \sum_{i=1}^r \left[ \sum_{j=1}^r (\Delta_U)_{2i-1,j} \right]\left[ \sum_{j=1}^r (\Delta_V)_{2i-1,j} \right] + \sum_{i=1}^r \Big\{\frac58 \left[ (\Delta_U)_{2i-1,i}^2 + (\Delta_V)_{2i-1,i}^2 \right]\\
    &+ \frac14(\Delta_U)_{2i-1,i}(\Delta_V)_{2i-1,i}+ \frac12\left[ (\Delta_U)_{2i,i} + (\Delta_V)_{2i,i} \right]^2 \Big\}\\
    & +  \sum_{1\leq i,j\leq n, i\neq j} \frac34\left[ (\Delta_U)_{2j,i} +(\Delta_V)_{2i,j} \right]^2+  \sum_{1\leq i,j\leq n, i\neq j} \frac34\left[ (\Delta_U)_{2j-1,i}^2 +(\Delta_V)_{2j-1,i}^2 \right]\\
    &=\sum_{i=1}^r \Big\{\frac58 \left[ (\Delta_U)_{2i-1,i} - (\Delta_V)_{2i-1,i} \right]^2 + \frac12\left[ (\Delta_U)_{2i,i} + (\Delta_V)_{2i,i} \right]^2 \Big\}\\
    & +  \sum_{1\leq i,j\leq n, i\neq j} \frac34\left[ (\Delta_U)_{2j,i} +(\Delta_V)_{2i,j} \right]^2 +  \sum_{1\leq i,j\leq n, i\neq j} \frac34\left[ (\Delta_U)_{2j-1,i} - (\Delta_V)_{2j-1,i} \right]^2,
\end{align*}
which is non-negative for every $\Delta\in\mathbb{R}^{n\times r}$. Hence, the point $\tilde{M}$ is a spurious second-order critical point of problem \eqref{eqn:asym}. Moreover, since we choose $\tilde{V} = \tilde{U}$, the point $\tilde{M}$ is a global minimizer of the regularizer $\|\tilde{U}^T\tilde{U}-\tilde{V}^T\tilde{U}\|_F^2$ and thus $\tilde{M}$ is also a spurious second-order critical point of problem \eqref{eqn:asym-reg}.
\end{example}

\section{Proofs for Section \ref{sec:strict-saddle} }

\subsection{Proof of Theorem \ref{thm:strict-saddle-1}}

In this subsection, we use the following notations:
\[ M := UV^T,~ M^* := U^*(V^*)^T,\quad W := \begin{bmatrix} U\\V \end{bmatrix},~ W^* := \begin{bmatrix} U^*\\V^* \end{bmatrix},\quad \hat{W} := \begin{bmatrix} U\\-V \end{bmatrix},~ \hat{W}^* := \begin{bmatrix} U^*\\-V^* \end{bmatrix}, \]
where $M^*:=M_a^*$ is the global optimum. We always assume that $U^*$ and $V^*$ satisfy $(U^*)^TU^* = (V^*)^TV^*$. When there is no ambiguity about $W$, we use $W^*$ to denote the minimizer of $\min_{X \in \mathcal{X}^*} \|W - X\|_F$, where $\mathcal{X}^*$ is the set of global minima of problem \eqref{eqn:asym-reg}. We note that the set $\mathcal{X}^*$ is the trajectory of a global minimum $(U^*,V^*)$ under the orthogonal group:
\[ \mathcal{X}^* = \{ (U^*R,V^*R)~|~ R\in\mathbb{R}^{r\times r}, R^TR = RR^T = I_r \}. \]
%
Therefore, the set $\mathcal{X}^*$ is a compact set and its minimum can be attained.
With this choice, it holds that
\[ \mathrm{dist}(W,\mathcal{X}^*) = \|W - W^*\|_F. \]
We first summarize some technical results in the following lemma.
\begin{lemma}[\cite{tu2016low,zhu2018global}]
\label{lem:tech}
The following statements hold for every $U \in \mathbb{R}^{n\times r}$, $V \in \mathbb{R}^{m\times r}$ and $W \in \mathbb{R}^{(n+m)\times r}$:
\begin{itemize}
    \item $4\|M-M^*\|_F^2 \geq \|WW^T - W^*(W^*)^T\|_F^2 - \|U^TU - V^TV\|_F^2$.
    \item $\|W^*(W^*)^T\|_F^2 = 4\|M^*\|_F^2$.
    \item If $\rank(W^*)=r$, then $\|WW^T - W^*(W^*)^T\|_F^2 \geq 2(\sqrt{2}-1)\sigma_r^2(W^*)\|W-W^*\|_F^2$.
    \item If $\rank(U^*)=r$, then $\|UU^T - U^*(U^*)^T\|_F^2 \geq 2(\sqrt{2}-1)\sigma_r^2(U^*)\|U-U^*\|_F^2$.
\end{itemize}
\end{lemma}

The proof of Theorem \ref{thm:strict-saddle-1} follows from the following sequence of lemmas. We first identify two cases when the gradient is large.
\begin{lemma}
Given a constant $\epsilon>0$, if 
\[ \|U^TU - V^TV\|_F \geq \epsilon, \]
then
\[ \|\nabla \rho(U,V)\|_F \geq \mu (\epsilon/r)^{3/2}. \]
\end{lemma}
\begin{proof}

Using the relationship between the $2$-norm and the Frobenius norm, we have
\[ \|U^TU - V^TV\|_2 \geq r^{-1}\|U^TU - V^TV\|_F \geq \epsilon / r. \]
Let $q\in\mathbb{R}^r$ be an eigenvector of $U^TU - V^TV$ such that
\[ \|q\|_2 = 1,\quad \left|q^T(U^TU - V^TV)q\right| = \|U^TU - V^TV\|_2. \]
We consider the direction
\[ \Delta := \hat{W} qq^T. \]
Then, we can calculate that
\[ \|\Delta\|_F^2 = \tr\left( \hat{W} qq^Tqq^T\hat{W}^T \right) =\tr\left( q^T\hat{W}^T\hat{W} q \right) = q^T(U^TU+V^TV)q. \]
In addition, we have
\begin{align*}
    \langle \nabla h_a(U,V) ,\Delta \rangle &= \left\langle \begin{bmatrix} \nabla f_a(M) V\\ \left[\nabla f_a(M)\right]^T U \end{bmatrix}, \begin{bmatrix} U qq^T\\ -Vqq^T \end{bmatrix} \right\rangle\\
    &= \tr\left[V^T [\nabla f_a(M)]^TU qq^T \right] - \tr\left[U^T \nabla f_a(M)V qq^T \right]\\
    &= q^T \left[ V^T [\nabla f_a(M)]^TU \right]q - q^T \left[ U^T \nabla f_a(M)V \right]q = 0.
\end{align*}
and
\begin{align*}
    \left|\left\langle \frac{\mu}{4}\nabla g(U,V), \Delta \right\rangle\right| &= \mu\left|\left\langle \hat{W}\hat{W}^T W, Wqq^T \right\rangle\right| \\
    &= \mu \left|\tr\left[ (U^TU-V^TV)(U^TU+V^TV)qq^T \right] \right|\\
    &= \mu \left| q^T(U^TU-V^TV)(U^TU+V^TV)q \right|\\
    &= \mu \|U^TU - V^TV\|_2 \cdot q^T(U^TU+V^TV)q\\
    &= \mu \|U^TU - V^TV\|_2 \cdot \sqrt{q^T(U^TU+V^TV)q} \cdot \|\Delta\|_F.
\end{align*}
Hence, Cauchy's inequality implies that
\begin{align*}
    \|\nabla \rho(U,V)\|_F &\geq \frac{\left|\langle \nabla \rho(U,V), \Delta \rangle\right|}{\|\Delta\|_F} = \mu \|U^TU - V^TV\|_2 \cdot \sqrt{q^T(U^TU+V^TV)q}.
\end{align*}
Using the fact that
\[ q^T(U^TU+V^TV)q \geq \left| q^T(U^TU-V^TV)q \right| = \|U^TU - V^TV\|_2, \]
we obtain
\[ \|\nabla \rho(U,V)\|_F \geq \mu \|U^TU - V^TV\|_2^{3/2} \geq \mu (\epsilon/r)^{3/2}.  \]
\end{proof}

\begin{lemma}
Given a constant $\epsilon>0$, if 
\[ \frac{1-\delta}{3} \leq \mu < 1-\delta,\quad  \|WW^T\|_F^{3/2} \geq \max\left\{\left(\frac{1+\delta}{1-\mu-\delta}\right)^2 \|W^*(W^*)^T\|_F^{3/2}, \frac{4\sqrt{r}\lambda}{1-\mu-\delta} \right\}, \]
then
\[ \|\nabla \rho(U,V)\|_F \geq \lambda. \]
\end{lemma}
\begin{proof}
Choosing the direction $\Delta := W$, we can calculate that
\begin{align}\label{eqn:5-5}
    \langle \nabla \rho(U,V), \Delta \rangle = 2\langle \nabla f_a(UV^T), UV^T\rangle + \mu \|U^TU - V^TV\|_F^2.
\end{align}
Using the $\delta$-RIP$_{2r,2r}$ property, we have
\begin{align*}
    [\nabla^2 f_a(N)](M,M) \geq (1-\delta) \|M\|_F^2,\quad [\nabla^2 f_a(N)](M^*,M) \leq (1+\delta) \|M\|_F\|M^*\|_F,
\end{align*}
where $N\in\mathbb{R}^{n\times m}$ is every matrix with rank at most $2r$. Then, the first term can be estimated as
\begin{align*}
    \langle \nabla f_a(UV^T), UV^T\rangle &= \int_0^1 [\nabla^2 f_a(M^* + s(M-M^*)][ M - M^*, M ]~ds\\
    &\geq (1-\delta)\|M\|_F^2 - (1+\delta) \|M^*\|_F \|M\|_F.
\end{align*}
The second term is 
\[ \mu\|U^TU - V^TV\|_F^2 = \mu \left(\|UU^T\|_F^2 + \|VV^T\|_F^2\right) - 2\mu\|M\|_F^2. \]
%
Substituting into equation \eqref{eqn:5-5}, it follows that
\begin{align*}
    \langle\nabla \rho(U,V), \Delta \rangle &\geq \mu\left(\|UU^T\|_F^2 + \|VV^T\|_F^2\right) + 2(1-\delta-\mu)\|M\|_F^2 - 2(1+\delta)\|M^*\|_F\|M\|_F\\
    &\geq \mu\left(\|UU^T\|_F^2 + \|VV^T\|_F^2\right) + 2(1-\delta-\mu)\|M\|_F^2 - 2c\|M\|_F^2 - \frac{(1+\delta)^2}{2c}\|M^*\|_F^2\\
    &\geq \min\left\{ \mu , 1-\delta-\mu-c  \right\} \|WW^T\|_F^2  - \frac{(1+\delta)^2}{2c}\|M^*\|_F^2,
\end{align*}
where $c\in(0,1-\delta-\mu)$ is a constant to be designed later. 
Using equality that $(U^*)^TU^* = (V^*)^TV^*$, Lemma \ref{lem:tech} gives
\[ \|W^*(W^*)^T\|_F^2 = 4\|M^*\|_F^2. \]
As a result, 
\begin{align*}
    \langle\nabla \rho(U,V), \Delta \rangle &\geq \min\left\{ \mu , 1-\delta-\mu-c  \right\} \|WW^T\|_F^2  - \frac{(1+\delta)^2}{8c}\|W^*(W^*)^T\|_F^2.
\end{align*}
Now, choosing
\[ c = \frac{1-\delta-\mu}{2} \]
and noticing that $\mu \geq (1-\delta-\mu)/2$, it yields that
\begin{align}\label{eqn:5-6}
    \langle\nabla \rho(U,V), \Delta \rangle &\geq \frac{1-\delta-\mu}{2} \|WW^T\|_F^2  - \frac{(1+\delta)^2}{4(1-\delta-\mu)}\|W^*(W^*)^T\|_F^2.
\end{align}
On the other hand,
\[ \|\Delta\|_F = \|W\|_F \leq \sqrt{r} \|WW^T\|_F^{1/2}. \]
Combining with inequality \eqref{eqn:5-6} and using the assumption of this lemma, one can write
\begin{align*}
    \|\nabla \rho(U,V)\|_F &\geq \frac{\langle\nabla \rho(U,V), \Delta \rangle}{\|\Delta\|_F} \\
    &\geq \frac{1-\delta-\mu}{2\sqrt{r}} \|WW^T\|_F^{3/2}  - \frac{(1+\delta)^2}{4\sqrt{r}(1-\delta-\mu)}\|W^*(W^*)^T\|_F^2 \|WW^T\|_F^{-1/2}\\
    &\geq \frac{1-\delta-\mu}{2\sqrt{r}} \|WW^T\|_F^{3/2}  - \frac{(1+\delta)^2}{4\sqrt{r}(1-\delta-\mu)}\|W^*(W^*)^T\|_F^{3/2}\\
    &\geq \frac{1-\delta-\mu}{4\sqrt{r}} \|WW^T\|_F^{3/2} \geq \lambda.
\end{align*}
\end{proof}

Using the above two lemmas, we only need to focus on points such that
\[ \|U^TU - V^TV\|_F = o(1),\quad \|WW^T\|_F = O(1). \]
The following lemma proves that if $(U,V)$ is an approximate first-order critical point with a small singular value $\sigma_r(W)$, then the Hessian of the objective function at this point has a negative curvature.
\begin{lemma}\label{lem:5-1}
Consider positive constants $\alpha,C,\epsilon,\lambda$ such that
\begin{align}\label{eqn:5-901} \epsilon^2 \leq (\sqrt{2}-1)\sigma_r^2(W^*) \cdot \alpha^2,\quad G > \mu \left( \epsilon + \frac{4H^2}{G^2}\right) + \frac{(1+\delta)H^2}{G^2} , \end{align}
where $G:= \|\nabla f_a(M)\|_2$ and $H:=\lambda + \mu\epsilon C$. If
\begin{align*} 
\|U^TU - V^TV\|_F^2 &\leq \epsilon^2,\quad \|WW^T\|_F \leq C^2,\quad \|W - W^*\|_F \geq \alpha,\quad \|\nabla \rho(U,V)\|_F \leq \lambda
\end{align*}
and
\begin{align}\label{eqn:cond} \sigma_r^2(W) \leq \frac{2}{1+\delta} \left[ G - \mu\left( \epsilon + \frac{4H^2}{G^2} \right) - \frac{(1+\delta)H^2}{G^2} \right] - 2\tau \end{align}
for some positive constant $\tau$, then it holds that
\[ \lambda_{min}(\nabla^2 \rho(U,V)) \leq - (1+\delta)\tau. \]
\end{lemma}
\begin{proof}

We choose a singular vector $q$ of $W$ such that
\[ \|q\|_2 = 1,\quad \|Wq\|_2 = \sigma_r(W). \]
Since $\|Wq\|_2 = \sqrt{\|Uq\|_2^2 + \|Vq\|_2^2}$, we have
\[ \|Uq\|_2^2 + \|Vq\|_2^2 = \sigma_r^2(W). \]
We choose singular vectors $u$ and $v$ such that
\[ \|u\|_2 = \|v\|_2 = 1,\quad \|\nabla f_a(M) \|_2 = u^T \nabla f_a(M) v. \]
We define the direction as
\[ \Delta_U := -uq^T,\quad \Delta_V := vq^T,\quad \Delta := \begin{bmatrix} \Delta_U \\ \Delta_V \end{bmatrix},\quad \hat{\Delta} := \begin{bmatrix} \Delta_U \\ -\Delta_V \end{bmatrix}. \]
For the Hessian of $h_a(\cdot,\cdot)$, we can calculate that
\begin{align}\label{eqn:5-8} \langle \nabla f_a(M) ,\Delta_U\Delta_V^T\rangle = -\|\nabla f_a(M)\|_2 = -G \end{align}
and the $\delta$-RIP$_{2r,2r}$ property gives
\begin{align}\label{eqn:5-9}
    \nonumber[\nabla ^2f_a(M)](\Delta_UV^T &+ U\Delta_V^T,\Delta_UV^T + U\Delta_V^T)\\
    \nonumber&\leq (1+\delta) \|\Delta_UV^T + U\Delta_V^T\|_F^2 = (1+\delta) \|-u (Vq)^T + (Uq)v^T\|_F^2\\
    \nonumber&= (1+\delta) \left( \|Vq\|_F^2 + \|Uq\|_F^2 \right) - 2(1+\delta) [ q^T (U^Tu)] \cdot [q^T (V^Tv) ]\\
    &\leq (1+\delta)\sigma_r^2(W) + 2(1+\delta) \cdot \|U^Tu\|_F  \|V^Tv\|_F .
\end{align}
Then, we consider the terms coming from the Hessian of the regularizer. First, we have
\begin{align}\label{eqn:5-10}
    \nonumber\langle\hat{\Delta}\hat{W}^T,\Delta W^T\rangle &\leq \|U^TU - V^TV\|_F \cdot \|\Delta_U^T\Delta_U - \Delta_V^T\Delta_V\|_F\\
    &\leq \epsilon \cdot \left[ \|\Delta_U^T\Delta_U\|_F + \|\Delta_V^T\Delta_V\|_F\right] = 2\epsilon.
\end{align} 
Next, we can estimate that
\begin{align}\label{eqn:5-7}
    \nonumber\langle\hat{W}\hat{\Delta}^T,\Delta W^T\rangle + \langle\hat{W}\hat{W}^T,{\Delta}\Delta^T\rangle &= \frac12 \|U^T\Delta_U + \Delta_U^TU - V^T\Delta_V - \Delta_V^T V \|_F^2\\
    \nonumber&\leq 4 \left( \|U^T\Delta_U\|_F^2 + \|V^T\Delta_V\|_F^2 \right)\\
    \nonumber&= 4\left( \|(U^Tu)q^T\|_F^2 + \|(V^Tv)q^T\|_F^2 \right)\\
    &= 4\left( \|U^Tu\|_F^2 + \|V^Tv\|_F^2 \right).
\end{align}
Using the assumption that $\|WW^T\|_F\leq C^2$ and $\|U^TU - V^TV\|_F^2 \leq \epsilon^2$, one can write 
\begin{align*}
    \|\hat{W}\hat{W}^TW\|_F^2 &\leq \|U^TU - V^TV\|_F^2 \cdot \|U^TU + V^TV\|_F \leq \epsilon^2 \|WW^T\|_F \leq \epsilon^2 C^2
\end{align*}
and
\begin{align}\label{eqn:5-101} 
\left\| \begin{bmatrix}\nabla f_a(UV^T)V\\ \nabla f_a(UV^T)^TU  \end{bmatrix} \right\|_F = \|\nabla \rho(U,V) - \mu\hat{W}\hat{W}^TW\|_F \leq \lambda + \mu\epsilon C = H.
\end{align}
The second relation implies that
\begin{align}\label{eqn:5-301} \|\nabla f_a(UV^T)V\|_2 \leq \|\nabla f_a(UV^T)V\|_F \leq H, \quad \|U^T \nabla f_a(UV^T)\|_2 \leq \|U^T \nabla f_a(UV^T)\|_F \leq H. \end{align}
By the definition of $u$ and $v$, it holds that
\[ \|v\|_2 = 1,\quad \|\nabla f_a(M)\|_2 u = \nabla f_a(M){v}. \]
Therefore,
\begin{align*} 
\|U^Tu\|_F^2 &= \frac{\|U^T \nabla f_a(M) {v}\|_F^2}{\|\nabla f_a(M)\|_2^2} \leq \frac{\|U^T \nabla f_a(M)\|_F^2 \|{v}\|_2^2}{\|\nabla f_a(M)\|_2^2} \leq \frac{H^2}{G^2}.
\end{align*}
Similarly, 
\[ \|V^Tv\|_F^2 \leq \frac{H^2}{G^2}. \]
Substituting into \eqref{eqn:5-9} and \eqref{eqn:5-7} yields that
\begin{align}\label{eqn:5-12}
    &\quad[\nabla ^2f_a(M)](\Delta_UV^T + U\Delta_V^T,\Delta_UV^T + U\Delta_V^T) \leq (1+\delta)\sigma_r^2(W) + 2(1+\delta) \cdot \frac{H^2}{G^2}
\end{align}
and
\begin{align}\label{eqn:5-11} \langle\hat{W}\hat{\Delta}^T,\Delta W^T\rangle + \langle\hat{W}\hat{W}^T,{\Delta}\Delta^T\rangle \leq 8 \cdot\frac{H^2}{G^2}. \end{align}
Combining \eqref{eqn:5-8}, \eqref{eqn:5-10}, \eqref{eqn:5-12} and \eqref{eqn:5-11}, it follows that
\begin{align*}
    [\nabla^2 \rho(U,V)](\Delta,\Delta) \leq -2G + (1+\delta)\sigma_r^2(W) + 2\mu\epsilon + [8\mu + 2(1+\delta)] \cdot\frac{H^2}{G^2}.
\end{align*}
Since $\|\Delta\|_F^2 = 2$, the above relation implies
\[ \lambda_{min}(\nabla^2 \rho(U,V)) \leq -G + \frac{1+\delta}{2}\sigma_r^2(W) + \mu \epsilon + (4\mu + 1+\delta) \cdot\frac{H^2}{G^2} \leq - (1+\delta) \tau. \]

\end{proof}
\begin{remark}
The positive constants  $\epsilon$ and $\lambda$ in the proof of Lemma \ref{lem:5-1} can be chosen to be arbitrarily small with $\alpha,C$ fixed. Hence, we may choose small enough $\epsilon$ and $\lambda$ such that the assumptions given in inequality \eqref{eqn:5-901} are satisfied. This lemma resolves the case when the minimal singular value $\sigma_r^2(W)$ is on the order of $\|\nabla f_a(M)\|_2 / (2+2\delta)$. In the next lemma, we will show that this is the only case when $\delta<1/3$.
\end{remark}

The final step is to prove that condition \eqref{eqn:cond} always holds provided that $\delta < 1/3$ and $\epsilon,\lambda, \tau =  o(1)$.
\begin{lemma}\label{lem:5-2}
Given positive constants $\alpha,C,\epsilon,\lambda$, if
\begin{align*} 
\|U^TU - V^TV\|_F^2 &\leq \epsilon^2,\quad \max\{\|WW^T\|_F,\|W^*(W^*)^T\|_F\} \leq C^2,\\
\|W - W^*\|_F &\geq \alpha,\quad \|\nabla \rho(U,V)\|_F \leq \lambda,\quad \delta < 1/3,
\end{align*}
then the inequality $G \geq c\alpha$ holds for some constant $c>0$ independent of $\alpha,\epsilon,\lambda,C$. Furthermore, there exist two positive constants
\[ \epsilon_0 (\delta,\mu,\sigma_r(M^*_a),\|M^*_a\|_F,\alpha,C),\quad \lambda_0 (\delta,\mu,\sigma_r(M^*_a),\|M^*_a\|_F,\alpha,C) \]
such that
\begin{align}\label{eqn:cond-1} \sigma_r^2(W) \leq \frac{2}{1+\delta} \left[ G - \mu\left( 2 \epsilon + \frac{4H^2}{G^2} \right) - \frac{(1+\delta)H^2}{G^2} \right] \end{align}
whenever 
\begin{align*}
    0 < &\epsilon \leq \epsilon_0 (\delta,\mu,\sigma_r(M^*_a),\|M^*_a\|_F,\alpha,C),\\
    0 < &\lambda \leq \lambda_0 (\delta,\mu,\sigma_r(M^*_a),\|M^*_a\|_F,\alpha,C).
\end{align*} 
Here, $G$ and $H$ are defined in Lemma \ref{lem:5-1}.
\end{lemma}
\begin{proof}

We first prove the existence of the constant $c$. Using Lemma \ref{lem:tech}, one can write 
\[ 4\|M-M^*\|_F^2 \geq \|WW^T - W^*(W^*)^T\|_F^2 - \|U^TU - V^TV\|_F^2 \geq \|WW^T - W^*(W^*)^T\|_F^2 - \epsilon^2. \]
Using Lemma \ref{lem:tech} and the assumption that $\|W-W^*\|_F \geq \alpha$, we have
\begin{align}\label{eqn:5-100} 
\|M-M^*\|_F^2 \geq \frac{\sqrt{2}-1}{2}\sigma_r^2(W^*)\|W-W^*\|_F^2 - \frac{\epsilon^2}{4} \geq \frac{\sqrt{2}-1}{2}\sigma_r^2(W^*) \cdot \alpha^2 - \frac{\epsilon^2}{4}. 
\end{align}
By the definition of $\epsilon$, it follows that
\[ \|M-M^*\|_F^2 \geq \frac{\sqrt{2}-1}{4}\sigma_r^2(W^*) \cdot \alpha^2 > 0. \]
Thus, the $\delta$-RIP$_{2r,2r}$ property gives
\[ \|\nabla f_a(M)\|_F \geq \frac{\langle \nabla f_a(M), M-M^*\rangle}{\|M-M^*\|_F} \geq (1-\delta) \|M-M^*\|_F \geq  \sqrt{\frac{\sqrt{2}-1}{4}} \cdot\sigma_r(W^*) (1-\delta) \cdot \alpha. \]
%
Hence, we have
\[ G = \|\nabla f_a(M) \|_2 \geq \sqrt{\frac{\sqrt{2}-1}{4r}} \cdot \sigma_r(W^*) (1-\delta) \cdot \alpha = c\alpha, \]
where we define
\[ c:= \sqrt{\frac{\sqrt{2}-1}{4r}} \cdot \sigma_r(W^*) (1-\delta). \]

Next, we prove inequality \eqref{eqn:cond-1} by contradiction, i.e., we assume
\begin{align}\label{eqn:cond-2} \sigma_r^2(W) > \frac{2}{1+\delta} \left[ G - \mu\left( 2 \epsilon + \frac{4H^2}{G^2} \right) - \frac{(1+\delta)H^2}{G^2} \right] \geq \frac{2c\alpha}{1+\delta}  + \mathrm{poly}(\epsilon,\lambda). \end{align}
The remainder of the proof is divided into three steps.
\paragraph{Step I.} We first develop a lower bound for $\sigma_r(M)$. 
We choose a vector $p\in\mathbb{R}^r$ such that
\[ \|p\|_F = 1,\quad U^TU p = \sigma_r^2(U) \cdot p. \]
It can be shown that 
\begin{align*}
    \|(Wp)^TW\|_F &= \|p^T U^TU + p^T V^TV\|_F \leq 2\|p^TU^TU\|_F + \|p^T(V^TV - U^TU)\|_F\\
    &\leq 2\sigma^2_r(U) + \|p^T\|_F\|V^TV - U^TU\|_F \leq 2\sigma^2_r(U) + \epsilon.
\end{align*}
On the other hand, since $W$ has rank $r$, it holds that
\[ \left\| (Wp)^T W\right\|_F \geq \sigma_r^2(W) \cdot \|p\|_F = \sigma_r^2(W). \]
Combining the above two estimates, we arrive at
\[ 2\sigma_r^2(U) \geq \sigma_r^2(W) - \epsilon > 0, \]
where the last inequality is from the assumption that $\epsilon,\lambda$ are small and $\sigma_r(W)$ is lower bounded by a positive value in \eqref{eqn:cond-2}. Using the inequality that $\sqrt{1-x} \geq 1-x$ for every $x\in[0,1]$, the above inequality implies that
\begin{align}\label{eqn:5-302} \sigma_r(U) \geq \frac{1}{\sqrt{2}}\sigma_r(W) \cdot \sqrt{1 - \frac{\epsilon}{\sigma_r^2(W)}} \geq \frac1{\sqrt{2}} \sigma_r(W) - \frac{\epsilon}{\sqrt{2}\sigma_r(W)}. \end{align}
Similarly, one can prove that
\[ \sigma_r(V) \geq \frac1{\sqrt{2}} \sigma_r(W) - \frac{\epsilon}{\sqrt{2}\sigma_r(W)}. \]
%
When $\epsilon$ is small enough, we know that $\sigma_r(U),\sigma_r(V) \neq 0$ and both $U,V$ have rank $r$. To lower bound the singular value $\sigma_r(M)$, we consider vectors $x$ such that $\|x\|_2 = 1$ and lower bound $x^T V (U^TU) V^T x$. Since the range of $V (U^TU) V^T$ is a subspace of the range of $V$ and the range of $V$ has exactly dimension $r$, directions $x$ that are in the orthogonal complement of the range of $V$ correspond to exactly $m-r$ zero singular values. Hence, to estimate the $r$-th largest singular value of $M$, we only need to consider directions that are in the range of $V$. Namely, we only consider directions that have the form $x = Vy$ for some vector $y$. Then, we have
\begin{align*} 
x^T V (U^TU) V^T x &= y^T (V^TV)(U^TU)(V^TV)y\\
&= y^T (V^TV)^3y + y^T (V^TV)(U^TU - V^TV)(V^TV)y. 
\end{align*}
First, we bound the second term by calculating that
\begin{align*}
    \|V(V^TV - U^TU)V^T\|_2 &\leq \|V\|_2^2 \|U^TU - V^TV\|_2 \leq \|V^TV\|_F \|U^TU - V^TV\|_F\\
    &\leq \|W^TW\|_F \|U^TU - V^TV\|_F \leq C^2 \epsilon.
\end{align*}
This implies that 
\[ y^T (V^TV)(U^TU - V^TV)(V^TV)y \geq -C^2\epsilon \cdot \|Vy\|_F^2. \]
Next, we assume that $y$ has the decomposition
\[ y = \sum_{i=1}^r c_i v_i, \]
where $v_i$ is an eigenvector of $V^TV$ associated with the eigenvalue $\sigma_i^2(V)$. 
Then, we can calculate that
\begin{align*}
    y^T (V^TV)^3y = \sum_{i=1}^r c_i^2 \sigma_i^6(V),\quad \|Vy\|_F^2 = \sum_{i=1}^r c_i^2 \sigma_i^2(V) = 1.
\end{align*}
Combining the above estimates leads to 
\begin{align*}
    x^T V (U^TU) V^T x &\geq \left[ \frac{\sum_{i=1}^r c_i^2 \sigma_i^6(V)}{\sum_{i=1}^r c_i^2 \sigma_i^2(V)} - C^2 \epsilon\right] \cdot \|Vy\|_F^2\\
    &= \frac{\sum_{i=1}^r c_i^2 \sigma_i^6(V)}{\sum_{i=1}^r c_i^2 \sigma_i^2(V)} - C^2\epsilon \geq \sigma_r^4(V) - C^2\epsilon.
\end{align*}
This implies that
\begin{align}\label{eqn:5-15}
    \nonumber\sigma_r^2(M) &\geq \sigma_r^4(V) - C^2\epsilon \geq \left[ \frac1{\sqrt{2}} \sigma_r(W) - \frac{\epsilon}{\sqrt{2}\sigma_r(W)} \right]^4 - C^2\epsilon\\
    \nonumber&\geq \frac{1}{4}\sigma_r^4(W) - \sigma_r^2(W)\epsilon - \sigma_r^{-2}(W)\epsilon^3 - C^2\epsilon\\
    \nonumber&\geq \frac{1}{4}\sigma_r^4(W) - \sigma_r^{-2}(W)\epsilon^3  - 2C^2\epsilon\\
    \nonumber&\geq \frac{1}{4}\sigma_r^4(W) - \frac{1+\delta}{G} \cdot \epsilon^3  - 2C^2\epsilon\\
    &\geq \frac{1}{4}\sigma_r^4(W) - \frac{1+\delta}{c\alpha} \cdot \epsilon^3  - 2C^2\epsilon.
\end{align}
where the second last inequality is due to \eqref{eqn:cond-2} and the assumption that $\epsilon$ and $\lambda$ are sufficiently small. 

\paragraph{Step II.} Next, we derive an upper bound for $\sigma_r(M)$. We define
\[ \bar{M} := \mathcal{P}_r\left[ M - \frac{1}{1+\delta} \nabla f_a(M) \right], \]
where $\mathcal{P}_r$ is the orthogonal projection onto the low-rank set via SVD. Since $M \neq M^*$ and $\delta < 1/3$, we recall that inequality \eqref{eqn:2-1} gives
\begin{align*}
    -\phi(\bar{M}) &\geq \frac{1-3\delta}{1-\delta} [ f_a(M) - f_a(M^*) ]  \geq  \frac{1-3\delta}{2}\|M-M^*\|_F^2\\
    &\geq \frac{1-3\delta}{2} \left[\frac{\sqrt{2}-1}{2}\sigma_r^2(W^*)\alpha^2 - \frac{\epsilon^2}{4} \right] := K,
\end{align*}
where the second inequality follows from \eqref{eqn:5-100} and
\[ -\phi(\bar{M}) = \langle \nabla f_a(M), M - \bar{M}\rangle - \frac{1+\delta}{2}\|M-\bar{M}\|_F^2. \]
Hence,
\begin{align}\label{eqn:5-13}
    \langle \nabla f_a(M), M - \bar{M}\rangle - \frac{1+\delta}{2}\|M-\bar{M}\|_F^2 \geq K.
\end{align}
When we choose $\epsilon$ to be small enough, it holds that $K > 0$. For simplicity, we define
\[ N := -\frac{1}{1 + \delta} \nabla f_a(M). \]
Then, $\bar{M} = \mathcal{P}_r(M+N)$ and the left-hand side of \eqref{eqn:5-13} is equal to
\begin{align}\label{eqn:5-14} 
\nonumber&\quad \langle \nabla f_a(M), M - \bar{M}\rangle - \frac{1+\delta}{2}\|M-\bar{M}\|_F^2\\
\nonumber&= (1+\delta) \langle N,  \mathcal{P}_r(M+N) - M \rangle - \frac{1+\delta}{2} \| \mathcal{P}_r(M+N) - M \|_F^2\\
\nonumber&= \frac{1+\delta}{2}\left[ \|N\|_F^2 - \| N + M - \mathcal{P}_r(M+N) \|_F^2 \right]\\
&= \frac{1+\delta}{2}\left[ \|N\|_F^2 - \| N + M\|_F^2 + \|\mathcal{P}_r(M+N) \|_F^2 \right].
\end{align}
Similar to the proof of inequality \eqref{eqn:5-301}, we can prove that
\[ \| NV \|_F \leq \tilde{H} := \frac{H}{1+\delta},\quad \|U^TN\|_F \leq \tilde{H}. \]
Then, we have
\begin{align*}
    -\tr[N^T (UV^T)] &\leq \|U^TN\|_F \|V\|_F \leq \tilde{H} \cdot \|W\|_F \leq \tilde{H} \cdot \sqrt{ \sqrt{r} \|WW^T\|_F } \leq \sqrt[4]{r}C \cdot \tilde{H}.
\end{align*}
Using the above relation, we obtain
\begin{align*}
    \|N\|_F^2 - \| N + M\|_F^2 &= -2\tr[ N^T (UV^T) ] - \|M\|_F^2 \leq 2 \sqrt[4]{r}C \cdot \tilde{H} - \|M\|_F^2.
\end{align*}
Suppose that $\mathcal{P}_U$ and $\mathcal{P}_V$ are the orthogonal projections onto the column spaces of $U$ and $V$, respectively. We define
\[ N_1 := \mathcal{P}_U N \mathcal{P}_V,~ N_2 := \mathcal{P}_U N (I-\mathcal{P}_V),~ N_3 := (I-\mathcal{P}_U)N\mathcal{P}_V,~ N_4 := (I-\mathcal{P}_U)N(I-\mathcal{P}_V). \]
Then, recalling the assumption \eqref{eqn:cond-2} and inequality \eqref{eqn:5-302}, it follows that
\begin{align*}
    \|N_1\|_F &= \| \mathcal{P}_U N \mathcal{P}_V \|_F \leq \sigma_r^{-1}(U) \| U^T \mathcal{P}_U N \mathcal{P}_V \|_F \leq \sigma_r^{-1}(U) \| U^T N \|_F \leq \frac{\sqrt{2}\sigma_r(W)}{\sigma_r^2(W) - \epsilon} \cdot \tilde{H}\\
    &\leq \left[\sqrt{\frac{1+\delta}{G}} + \mathrm{poly}(\epsilon,\lambda) \right] \cdot \tilde{H} \leq \left[\sqrt{\frac{1+\delta}{c\alpha}} + \mathrm{poly}(\epsilon,\lambda) \right] \cdot \tilde{H} := \kappa \tilde{H}.
\end{align*}
Similarly, we can prove that
\begin{align*}
    &\|N_1 + N_2\|_F = \| \mathcal{P}_U N\|_F \leq \kappa\tilde{H},\quad \|N_1 + N_3\|_F = \| N \mathcal{P}_V \|_F \leq \kappa \tilde{H},
\end{align*}
which leads to
\[ \|N_2\|_F \leq 2\kappa \tilde{H},\quad \|N_3\|_F \leq 2\kappa \tilde{H}. \]
Using Weyl's theorem, the following holds for every $1\leq i\leq r$:
\begin{align*}
    | \sigma_i(M+N) - \sigma_i(M+N_4) | \leq \|N_1+N_2+N_3\|_2 \leq \|N_1+N_2+N_3\|_F \leq 3\kappa \tilde{H}.
\end{align*}
Therefore, we have
\begin{align}\label{eqn:5-102}
    \nonumber\|\mathcal{P}_r(M+N) \|_F^2 &= \sum_{i=1}^r \sigma_i^2(M+N)\\
    \nonumber&\geq \sum_{i=1}^r \sigma_i^2(M+N_4) - r \cdot 3\kappa \tilde{H} \cdot ( \|M+N\|_2 + \|M+N_4\|_2 )\\
    \nonumber&\geq \sum_{i=1}^r \sigma_i^2(M+N_4) - 6r \kappa \tilde{H} \cdot ( \|M\|_2 + \|N\|_2 )\\
    &\geq \sum_{i=1}^r \sigma_i^2(M+N_4) - 6r \kappa \tilde{H} \cdot \left( \|M\|_F  + \frac{G}{1+\delta} \right).
\end{align}
%
Using the assumption \eqref{eqn:cond-2} and the inequality \eqref{eqn:5-15}, one can write 
\begin{align}\label{eqn:5-303}
    \frac{G}{1+\delta} &\leq \frac{\sigma_r^2(W)}{2} + \mathrm{poly}(\sqrt{\epsilon},\lambda) \leq \sigma_r(M) + \mathrm{poly}(\sqrt{\epsilon},\lambda) \leq \|M\|_F + \mathrm{poly}(\sqrt{\epsilon},\lambda),
\end{align}
where $\mathrm{poly}(\sqrt{\epsilon},\lambda)$ means a polynomial of $\sqrt{\epsilon}$ and $\lambda$. Therefore, we attain the bound
\begin{align}\label{eqn:5-304} 
\nonumber\|M\|_F  + \|N\|_F &\leq 2\|M\|_F + \mathrm{poly}(\sqrt{\epsilon},\lambda) \leq 2 \cdot \frac{\|WW^T\|_F}{\sqrt{2}} + \mathrm{poly}(\sqrt{\epsilon},\lambda)\\
&\leq \sqrt{2}C^2 + \mathrm{poly}(\sqrt{\epsilon},\lambda). 
\end{align}
Substituting back into the previous estimate \eqref{eqn:5-102}, it follows that
\[ \|\mathcal{P}_r(M+N) \|_F^2 \geq \sum_{i=1}^r \sigma_i^2(M+N_4) - 6\sqrt{2} r \kappa \tilde{H}C^2 + \mathrm{poly}(\sqrt{\epsilon},\lambda) = \sum_{i=1}^r \sigma_i^2(M+N_4) + \mathrm{poly}(\sqrt{\epsilon},\lambda). \]
Now, since $M$ and $N_4$ have orthogonal column and row spaces, the maximal $r$ singular values of $M+N_4$ are simply the maximal $r$ singular values of the singular values $M$ and $N_4$, which we assume to be
\[ \sigma_i(M),~i=1,\dots,k \quad \text{and} \quad \sigma_i(N_4),~i=1,\dots,r-k. \]
Now, it follows from \eqref{eqn:5-14} that
\begin{align*}
    &\hspace{-3em}\frac{2}{1+\delta}\left[\langle \nabla f_a(M), M - \bar{M}\rangle - \frac{1+\delta}{2}\|M-\bar{M}\|_F^2\right]\\
    &=\|N\|_F^2 - \| N + M\|_F^2 + \|\mathcal{P}_r(M+N) \|_F^2\\
    &\leq -\sum_{i=1}^r \sigma_i^2(M) + \sum_{i=1}^k \sigma_i^2(M) + \sum_{i=1}^{r-k} \sigma_i^2(N_4) + \mathrm{poly}(\sqrt{\epsilon},\lambda) + 2 \sqrt[4]{r}C \cdot \tilde{H}\\
    &= -\sum_{i=k+1}^{r} \sigma_i^2(M) + \sum_{i=1}^{r-k} \sigma_i^2(N_4) + \mathrm{poly}(\sqrt{\epsilon},\lambda)\\
    &\leq -(r-k) \sigma_r^2(M) + (r-k)\|N_4\|_2^2 + \mathrm{poly}(\sqrt{\epsilon},\lambda)\\
    &\leq -(r-k) \sigma_r^2(M) + (r-k)\|N\|_2^2 + \mathrm{poly}(\sqrt{\epsilon},\lambda).
\end{align*}
If $k=r$, then the above inequality and inequality \eqref{eqn:5-13} imply that
\[ \mathrm{poly}(\sqrt{\epsilon},\lambda) \geq K = O(\alpha^2), \]
which contradicts the assumption that $\epsilon$ and $\lambda$ are small. Hence, it can be concluded that  $r-k\geq1$. Combining with \eqref{eqn:5-13}, we obtain the upper bound
\begin{align}\label{eqn:5-16}
    \nonumber\sigma_r^2(M) &\leq - \frac{2}{1+\delta} \cdot \frac{K}{r-k} + \|N\|_2^2 + \frac{1}{r-k} \cdot \mathrm{poly}(\sqrt{\epsilon},\lambda)\\
    &= - \frac{2}{1+\delta}\cdot \frac{K}{r} + \|N\|_2^2 + \mathrm{poly}(\sqrt{\epsilon},\lambda).
\end{align}

\paragraph{Step III.} In the last step, we combine the inequalities \eqref{eqn:5-15} and \eqref{eqn:5-16}, which leads to
\begin{align*}
    \frac{1}{4}\sigma_r^4(W) - \frac{1+\delta}{c\alpha} \cdot \epsilon^3  - 2C^2\epsilon \leq - \frac{2}{1+\delta}\cdot \frac{K}{r} + \frac{1}{(1+\delta)^2}G^2 +  \mathrm{poly}(\sqrt{\epsilon},\lambda).
\end{align*}
This means that
\begin{align*}
    \sigma_r^4(W) + \frac{8}{1+\delta}\cdot \frac{K}{r} \leq \frac{4}{(1+\delta)^2}G^2 + \mathrm{poly}(\sqrt{\epsilon},\lambda).
\end{align*}
%
Since $K>0$ has lower bounds that are independent of $\epsilon$ and $\lambda$, we can choose $\epsilon$ and $\lambda$ to be small enough such that
\begin{align*}
    \sigma_r^4(W) + \frac{4}{1+\delta}\cdot \frac{K}{r} \leq \frac{4}{(1+\delta)^2}G^2.
\end{align*}
However, recalling the assumption \eqref{eqn:cond-2}, we have
\begin{align*}
    \sigma_r^4(W) &> \frac{4}{(1+\delta)^2} \left[ G - \mu\left( 2 \epsilon + \frac{4H^2}{G^2} \right) - \frac{(1+\delta)H^2}{G^2} \right]^2\\
    &\geq \frac{4}{(1+\delta)^2} G^2 - \frac{16}{(1+\delta)^2} G \cdot \mu\epsilon + \mathrm{poly}(\sqrt{\epsilon},\lambda)\\
    &\geq \frac{4}{(1+\delta)^2} G^2 - \frac{16}{(1+\delta)^2} \mu\epsilon \cdot \frac{1}{\sqrt{2}}(1+\delta) C^2 + \mathrm{poly}(\sqrt{\epsilon},\lambda)\\
    &= \frac{4}{(1+\delta)^2} G^2 + \mathrm{poly}(\sqrt{\epsilon},\lambda),
\end{align*}
where in the third inequality we use inequalities \eqref{eqn:5-303}-\eqref{eqn:5-304} to conclude that
\[ G \leq (1+\delta)\|M\|_F + \mathrm{poly}(\sqrt{\epsilon},\lambda) \leq \frac{1}{\sqrt{2}}(1+\delta)C^2 + \mathrm{poly}(\sqrt{\epsilon},\lambda). \]
The above two inequalities cannot hold simultaneously when $\lambda$ and $\epsilon$ are small enough. This contradiction means that the condition \eqref{eqn:cond-1} holds by choosing 
\begin{align*} 
0<&\epsilon \leq \epsilon_0 (\delta,\mu,\sigma_r(M^*_a),\|M^*_a\|_F,\alpha,C),\\
0<&\lambda \leq \lambda_0 (\delta,\mu,\sigma_r(M^*_a),\|M^*_a\|_F,\alpha,C),
\end{align*}
for some small enough positive constants
\[ \epsilon_0 (\delta,\mu,\sigma_r(M^*_a),\|M^*_a\|_F,\alpha,C),\quad \lambda_0 (\delta,\mu,\sigma_r(M^*_a),\|M^*_a\|_F,\alpha,C). \]

\end{proof}

The only thing left is to piecing everything together.
\begin{proof}[Proof of Theorem \ref{thm:strict-saddle-1}]
We first choose 
\[ C:= \left[\left(\frac{1+\delta}{1-\mu-\delta}\right)^2 \|W^*(W^*)^T\|_F^{3/2} \right]^{1/3}. \]
Then, we select $\epsilon_1$ and $\lambda_1$ as 
\begin{align*} 
\epsilon_1(\delta,r,\mu,\sigma_r(M^*_a),\|M^*_a\|_F,\alpha) &:= \epsilon_0(\delta,r,\mu,\sigma_r(M^*_a),\|M^*_a\|_F,\alpha,C),\\
\lambda_1(\delta,r,\mu,\sigma_r(M^*_a),\|M^*_a\|_F,\alpha) &:= \min\Bigg\{\lambda_0(\delta,r,\mu,\sigma_r(M^*_a),\|M^*_a\|_F,\alpha,C),\\
&\hspace{14em}\frac{(1-\mu-\delta)C^3}{4\sqrt{r}}\Bigg\}.
\end{align*}
Finally, we combine Lemmas \ref{lem:5-1}-\ref{lem:5-2} to get the bounds for the gradient and the Hessian.
\end{proof}

\subsection{Proof of Theorem \ref{thm:strict-saddle-2}}

In this subsection, we use similar notations:
\[ M := UU^T,\quad M^* := U^*(U^*)^T, \]
where  $M^*:=M_s^*$ is the global optimum. We also assume that $U^*$ is the minimizer of $\min_{X \in \mathcal{X}^*}\|U-X\|_F$ when there is no ambiguity about $U$. In this case, the distance is given by
\[ \mathrm{dist}(U,\mathcal{X}^*) = \|U - U^*\|_F. \]
The proof of Theorem \ref{thm:strict-saddle-2} is similar to that of Theorem \ref{thm:strict-saddle-1}. We first consider the case when $\|UU^T\|_F$ is large.
\begin{lemma}\label{lem:7-1}
Given a constant $\epsilon>0$, if 
\[ \|UU^T\|_F^{2} \geq \max\left\{\frac{2(1+\delta)}{1-\delta} \|U^*(U^*)^T\|_F^{2}, \left(\frac{2\lambda\sqrt{r}}{1-\delta}\right)^{4/3}  \right\}, \]
then
\[ \|\nabla h_s(U)\|_F \geq \lambda. \]
\end{lemma}
\begin{proof}
Choosing the direction $\Delta := U$, we can calculate that
\begin{align*}
    \langle \nabla h_s(U), \Delta \rangle = \langle \nabla f_s(UU^T), UU^T\rangle.
\end{align*}
Using the $\delta$-RIP$_{2r,2r}$ property, we have
\begin{align*}
    \langle \nabla f_s(UU^T), UU^T\rangle &= \int_0^1 [\nabla^2 f_s(M^* + s(M-M^*)][ M - M^*, M ]\\
    &\geq (1-\delta)\|M\|_F^2 - (1+\delta) \|M^*\|_F \|M\|_F\\
    &\geq \frac{1-\delta}{2}\|M\|_F^2.
\end{align*}
Moreover, 
\[ \|\Delta\|_F = \|U\|_F \leq \sqrt{r} \|UU^T\|_F^{1/2}. \]
This leads to
\begin{align*}
    \| \nabla h_s(U) \|_F &\geq \frac{\langle \nabla h_s(U), \Delta\rangle}{\|\Delta\|_F} = \frac{\langle \nabla f_s(UU^T), UU^T\rangle}{\|U\|_F} \geq \frac{1-\delta}{2\sqrt{r}}\|UU^T\|_F^{3/2} \geq \lambda.
\end{align*}
\end{proof}

The next lemma is a counterpart of Lemma \ref{lem:5-1}.
\begin{lemma}
Consider positive constants  $\alpha,C,\lambda$ such that
\[ \lambda \leq 2(\sqrt{r}C)^{-1}(\sqrt{2}-1)\sigma_r^2(U^*) \cdot \alpha^2,\quad G > \frac{(1+\delta)\lambda^2}{4G^2}, \]
where $G:= -\lambda_{min}(\nabla f_s(M))$. If
\begin{align*} 
 \|UU^T\|_F \leq C^2,\quad \|U - U^*\|_F \geq \alpha,\quad \|\nabla h_s(U)\|_F \leq \lambda,
\end{align*}
then the inequality $G \geq c\alpha^2$ holds for some constant $c>0$ independent of $\alpha,\lambda,C$. Moreover, if 
\begin{align}\label{eqn:cond-sym} \sigma_r^2(U) \leq \frac{1}{1+\delta} \left[ G - \frac{(1+\delta)\lambda^2}{4G^2} \right] - \tau, \end{align}
then there exists some positive constant $\tau$ such that 
\[ \lambda_{min}(\nabla^2 h_s(U)) \leq - 2(1+\delta)\tau. \]
\end{lemma}
\begin{proof}

We choose a singular vector $q$ of $U$ such that
\[ \|q\|_2 = 1,\quad \|Uq\|_2 = \sigma_r(U). \]
We first prove the existence of the constant $c$. The $\delta$-RIP$_{2r,2r}$ property gives
\[ \langle \nabla f_s(M), M^*-M\rangle \leq -(1-\delta) \|M-M^*\|_F^2. \]
Using the assumption of this lemma, we have
\begin{align}\label{eqn:7-101} 
\|\nabla f_s(M) U\|_2 \leq \left\|  \nabla f_s(M) U \right\|_F = \frac12 \|\nabla h_s(U)\|_F \leq \frac12 \lambda,
\end{align}
which leads to
\[ \langle \nabla f_s(M), M\rangle = \langle\nabla f_s(M) U, U\rangle \leq \|\nabla f_s(M)U\|_F\|U\|_F \leq \frac12 \lambda \cdot \sqrt{r}C. \]
Substituting into \eqref{eqn:7-101}, it follows that
\[ \langle \nabla f_s(M), M^*\rangle \leq -(1-\delta) \|M-M^*\|_F^2 + \frac12 \lambda \cdot \sqrt{r}C. \]
Using Lemma \ref{lem:tech}, we have
\begin{align*}
\|M-M^*\|_F^2 \geq 2(\sqrt{2}-1)\sigma_r^2(U^*)\|U-U^*\|_F^2  \geq 2(\sqrt{2}-1)\sigma_r^2(U^*) \cdot \alpha^2. 
\end{align*}
By the condition on $\lambda$, it follows that
\begin{align}\label{eqn:7-105} \langle \nabla f_s(M), M^*\rangle \leq -(1-\delta) \|M-M^*\|_F^2 + \frac12 \lambda \cdot \sqrt{r}C \leq -(\sqrt{2}-1)(1-\delta) \sigma_r^2(U^*) \cdot \alpha^2. \end{align}
The above inequality also indicates that $\lambda_{min}(\nabla f_s(M)) < 0$. Using the relations that
\[ \nabla f_s(M) \succeq \lambda_{min}(\nabla f_s(M)) \cdot I_n,\quad M^*\succeq0, \]
we arrive at
\[ \langle \nabla f_s(M), M^*\rangle \geq \lambda_{min}(\nabla f_s(M)) \tr(M^*) \geq \sqrt{r}\|M^*\|_F \cdot \lambda_{min}(\nabla f_s(M)). \]
Combining the last inequality with \eqref{eqn:7-105}, we obtain
\begin{align*}
    \lambda_{min}(\nabla f_s(M)) \leq - (\sqrt{r}\|M^*\|_F)^{-1} (\sqrt{2}-1)(1-\delta) \sigma_r^2(U^*) \cdot \alpha^2 = -c\alpha^2
\end{align*}
and thus $G \geq c \alpha^2$, where
\[ c := (\sqrt{r}\|M^*\|_F)^{-1} (\sqrt{2}-1)(1-\delta) \sigma_r^2(U^*) \]

Next, we prove the upper bound on the minimal eigenvalue. We choose an eigenvector $u$ such that
\[ \|u\|_2 = 1,\quad \lambda_{min}(\nabla f_s(M)) = u^T \nabla f_s(M) u. \]
The direction is chosen to be
\[ \Delta := uq^T. \]
For the Hessian of $h_s(\cdot,\cdot)$, we can calculate that
\begin{align}\label{eqn:7-8} \langle \nabla f_s(M) ,\Delta\Delta ^T\rangle = \lambda_{min}(\nabla f_s(M)) = -G \end{align}
and the $\delta$-RIP$_{2r,2r}$ property gives
\begin{align}\label{eqn:7-9}
    \nonumber[\nabla ^2f_s(M)](\Delta U^T + &U\Delta^T,\Delta U^T + U\Delta^T)\\
    \nonumber&\leq (1+\delta) \|\Delta U^T + U\Delta^T\|_F^2 = (1+\delta) \| u (Uq)^T + (Uq)u^T\|_F^2\\
    \nonumber&= 2(1+\delta) \|Uq\|_F^2 + 2(1+\delta) [ q^T (U^Tu)]^2\\
    &\leq 2(1+\delta)\sigma_r^2(U) + 2 (1+\delta) \cdot \|U^Tu\|_F^2 .
\end{align}
%
%
By letting the vector $\tilde{v}$ be
\[ \|\tilde{v}\|_2 = 1,\quad \lambda_{min}(\nabla f_s(M)) u = \nabla f_s(M) \tilde{v}, \]
the inequality \eqref{eqn:7-101} implies that
\begin{align*} 
\|U^Tu\|_F^2 &= \frac{\|U^T \nabla f_s(M) \tilde{v}\|_F^2}{\lambda_{min}^2(\nabla f_s(M))} = \frac{\|U^T \nabla f_s(M) \tilde{v}\|_2^2}{\lambda_{min}^2(\nabla f_s(M))} \leq \frac{\|U^T \nabla f_s(M)\|_2^2 \|\tilde{v}\|_2^2}{\lambda_{min}^2(\nabla f_s(M))} \leq \frac{\lambda^2}{4G^2}.
\end{align*}
%
Substituting into \eqref{eqn:7-9}, we obtain
\begin{align}\label{eqn:7-12}
    &\quad[\nabla ^2f_s(M)](\Delta U^T + U\Delta^T,\Delta U^T + U\Delta^T) \leq 2(1+\delta)\sigma_r^2(U) + (1+\delta) \cdot \frac{\lambda^2}{2G^2}.
\end{align}
%
Combining \eqref{eqn:7-8} and \eqref{eqn:7-12}, it follows that
\begin{align*}
    [\nabla^2 h_s(U)](\Delta,\Delta) \leq -2G + 2(1+\delta)\sigma_r^2(U) + (1+\delta) \cdot \frac{\lambda^2}{2G^2}.
\end{align*}
Since $\|\Delta\|_F^2 = 1$, the above inequality implies
\[ \lambda_{min}(\nabla^2 h_s(U)) \leq -2G + 2(1+\delta)\sigma_r^2(U) + (1+\delta) \cdot \frac{\lambda^2}{2G^2} \leq - (1+\delta) \tau. \]

\end{proof}

We finally give the counterpart of Lemma \ref{lem:5-2}, which states that the condition \eqref{eqn:cond-sym} always holds when $\delta < 1/3$.
\begin{lemma}\label{lem:7-2}
Given positive constants $\alpha,C,\epsilon,\lambda$, if
\begin{align*} 
\max\{\|UU^T\|_F,\|U^*(U^*)^T\|_F\} \leq C^2,~\|U - U^*\|_F &\geq \alpha,~ \|\nabla h_s(U)\|_F \leq \lambda,~ \delta < 1/3,
\end{align*}
then there exists a positive constant $\lambda_0 (\delta,W^*,\alpha,C)$ such that
\begin{align}\label{eqn:cond-sym-1} \sigma_r^2(U) \leq \frac{1}{1+\delta} \left[ G - \frac{(1+\delta)\lambda^2}{4G^2} - \lambda \right] \end{align}
whenever 
\[ 0 < \lambda \leq \lambda_0 (\delta,\sigma_r(M^*_s),\|M^*_s\|_F,\alpha,C). \]
\end{lemma}
\begin{proof}

We prove by contradiction, i.e., we assume
\begin{align}\label{eqn:cond-sym-2} \sigma_r^2(U) > \frac{1}{1+\delta} \left[ G - \frac{(1+\delta)\lambda^2}{4G^2} - \lambda \right] \geq \frac{c\alpha^2}{1+\delta} + \mathrm{poly}(\lambda). \end{align}
To follow the proof of Lemma \ref{lem:5-2}, we also divide the argument into three steps, although the first step is superficial.
\paragraph{Step I.} We first give a lower bound for $\lambda_r(M)$. In the symmetric case, this step is straightforward, since we always have
\begin{align}\label{eqn:7-15} \lambda_r^2(M) = \sigma_r^4(U). \end{align}

\paragraph{Step II.} Next, we derive an upper bound for $\lambda_r(M)$. We define
\[ \bar{M} := \mathcal{P}_r\left[ M - \frac{1}{1+\delta} \nabla f_s(M) \right], \]
where $\mathcal{P}_r$ is the orthogonal projection onto the low-rank manifold (we do not drop negative eigenvalues in this proof). Since $M \neq M^*$ and $\delta < 1/3$, we recall that inequality \eqref{eqn:2-1} gives
\begin{align*}
    -\phi(\bar{M}) &\geq \frac{1-3\delta}{1-\delta} [ f_s(M) - f_s(M^*) ]  \geq  \frac{1-3\delta}{2}\|M-M^*\|_F^2\\
    &\geq (1-3\delta) \cdot (\sqrt{2}-1)\sigma_r^2(W^*)\alpha^2  := K > 0,
\end{align*}
where the second inequality comes from Lemma \ref{lem:tech} and
\[ -\phi(\bar{M}) = \langle \nabla f_s(M), M - \bar{M}\rangle - \frac{1+\delta}{2}\|M-\bar{M}\|_F^2. \]
Hence, 
\begin{align}\label{eqn:7-13}
    \langle \nabla f_s(M), M - \bar{M}\rangle - \frac{1+\delta}{2}\|M-\bar{M}\|_F^2 \geq K.
\end{align}
For simplicity, we define
\[ N := -\frac{1}{1 + \delta} \nabla f_s(M). \]
Then, $\bar{M} = \mathcal{P}_r(M+N)$ and the left-hand side of \eqref{eqn:7-13} is equal to
\begin{align}\label{eqn:7-14} 
\nonumber \langle \nabla f_s(M), M - \bar{M}\rangle &- \frac{1+\delta}{2}\|M-\bar{M}\|_F^2\\
\nonumber&= (1+\delta) \langle N,  \mathcal{P}_r(M+N) - M \rangle - \frac{1+\delta}{2} \| \mathcal{P}_r(M+N) - M \|_F^2\\
\nonumber&= \frac{1+\delta}{2}\left[ \|N\|_F^2 - \| N + M - \mathcal{P}_r(M+N) \|_F^2 \right]\\
&= \frac{1+\delta}{2}\left[ \|N\|_F^2 - \| N + M\|_F^2 + \|\mathcal{P}_r(M+N) \|_F^2 \right].
\end{align}
Similar to the proof of inequality \eqref{eqn:7-101}, we can prove that
\[ \|U^TN\|_F \leq \tilde{H} := \frac{\lambda}{2(1+\delta)}. \]
Then, we have
\begin{align*}
    -\tr[N^T (UU^T)] &\leq \|U^TN\|_F \|U\|_F \leq \tilde{H} \cdot \|U\|_F \leq \tilde{H} \cdot \sqrt{ \sqrt{r} \|UU^T\|_F } \leq \sqrt[4]{r}C \cdot \tilde{H}.
\end{align*}
Using the above relation, one can write
\begin{align*}
    \|N\|_F^2 - \| N + M\|_F^2 &= -2\tr[ N^T (UU^T) ] - \|M\|_F^2 \leq 2 \sqrt[4]{r}C \cdot \tilde{H} - \|M\|_F^2.
\end{align*}
Suppose that $\mathcal{P}_U$ is the orthogonal projections onto the column space of $U$. We define
\[ N_1 := \mathcal{P}_U N \mathcal{P}_U,~ N_2 := \mathcal{P}_U N (I-\mathcal{P}_U),~ N_3 := (I-\mathcal{P}_U)N\mathcal{P}_U,~ N_4 := (I-\mathcal{P}_U)N(I-\mathcal{P}_U). \]
Then, it follows from \eqref{eqn:cond-sym-2} that
\begin{align*}
    \|N_1\|_F &= \| \mathcal{P}_U N \mathcal{P}_U \|_F \leq \sigma_r^{-1}(U) \| U^T \mathcal{P}_U N \mathcal{P}_U \|_F \leq \sigma_r^{-1}(U) \| U^T N \|_F \leq \sigma_r^{-1}(U) \cdot \tilde{H}\\
    &\leq \left[\sqrt{\frac{1+\delta}{G}} + \mathrm{poly}(\lambda) \right] \cdot \tilde{H} \leq \left[\sqrt{\frac{1+\delta}{c\alpha^2}} + \mathrm{poly}(\lambda) \right] \cdot \tilde{H} := \kappa \tilde{H}.
\end{align*}
Similarly, we can prove that
\begin{align*}
    &\|N_1 + N_2\|_F = \| \mathcal{P}_U N\|_F \leq \kappa\tilde{H},\quad \|N_1 + N_3\|_F = \| N \mathcal{P}_V \|_F \leq \kappa \tilde{H},
\end{align*}
which leads to
\[ \|N_2\|_F \leq 2\kappa \tilde{H},\quad \|N_3\|_F \leq 2\kappa \tilde{H}. \]
Using Weyl's theorem, the following holds for every $1\leq i\leq r$:
\begin{align*}
    | \lambda_i(M+N) - \lambda_i(M+N_4) | \leq \|N_1+N_2+N_3\|_2 \leq \|N_1+N_2+N_3\|_F \leq 3\kappa \tilde{H}.
\end{align*}
Therefore, we have
\begin{align}\label{eqn:7-102}
    \nonumber\|\mathcal{P}_r(M+N) \|_F^2 &= \sum_{i=1}^r \lambda_i^2(M+N)\\
    \nonumber&\geq \sum_{i=1}^r \lambda_i^2(M+N_4) - r \cdot 3\kappa \tilde{H} \cdot ( \|M+N\|_2 + \|M+N_4\|_2 )\\
    \nonumber&\geq \sum_{i=1}^r \lambda_i^2(M+N_4) - 6r \kappa \tilde{H} \cdot ( \|M\|_2 + \|N\|_2 )\\
    &\geq \sum_{i=1}^r \lambda_i^2(M+N_4) - 6r \kappa \tilde{H} \cdot \left( \|M\|_F  + \frac{G}{1+\delta} \right).
\end{align}
%
Similar to the asymmetric case, we can prove that
\[ \frac{G}{1+\delta} \leq \|M\|_F + \mathrm{poly}(\lambda). \]
holds under the assumption \eqref{eqn:cond-sym-2}. Therefore, we obtain the bound
\[ \|M\|_F  + \|N\|_F \leq 2\|M\|_F + \mathrm{poly}(\lambda) \leq 2C^2 + \mathrm{poly}(\lambda). \]
Substituting back into the previous estimate \eqref{eqn:7-102}, it follows that
\[ \|\mathcal{P}_r(M+N) \|_F^2 \geq \sum_{i=1}^r \lambda_i^2(M+N_4) + \mathrm{poly}(\lambda). \]
Now, since $M$ and $N_4$ have orthogonal column and row spaces, the maximal $r$ eigenvalues of $M+N_4$ are simply the maximal $r$ eigenvalues of the eigenvalues of $M$ and $N_4$, which we assume to be
\[ \lambda_i(M),~i=1,\dots,k \quad \text{and} \quad \lambda_i(N_4),~i=1,\dots,r-k. \]
Now, it follows from \eqref{eqn:7-14} that
\begin{align}\label{eqn:7-103}
    \nonumber&\hspace{-5em}\frac{2}{1+\delta}\left[\langle \nabla f_s(M), M - \bar{M}\rangle - \frac{1+\delta}{2}\|M-\bar{M}\|_F^2\right]\\
    \nonumber&=\|N\|_F^2 - \| N + M\|_F^2 + \|\mathcal{P}_r(M+N) \|_F^2\\
    \nonumber&\leq -\sum_{i=1}^r \lambda_i^2(M) + \sum_{i=1}^k \lambda_i^2(M) + \sum_{i=1}^{r-k} \lambda_i^2(N_4) + \mathrm{poly}(\lambda) + 2 \sqrt[4]{r}C \cdot \tilde{H}\\
    &= -\sum_{i=k+1}^{r} \lambda_i^2(M) + \sum_{i=1}^{r-k} \lambda_i^2(N_4) + \mathrm{poly}(\lambda).
\end{align}
Using the assumption \eqref{eqn:cond-sym-2} and the fact that $\lambda$ is small, we know that $\lambda_i(N_4) > 0$ for all $i\in\{1,\dots,k\}$. Therefore,
\[ -\sum_{i=k+1}^{r} \lambda_i^2(M) + \sum_{i=1}^{r-k} \lambda_i^2(N_4) \leq -(r-k) \lambda_r^2(M) + (r-k)\lambda_{max}(N_4)^2. \]
Substituting into \eqref{eqn:7-103} gives rise to 
\begin{align*}
    &\hspace{-8em}\frac{2}{1+\delta}\left[\langle \nabla f_s(M), M - \bar{M}\rangle - \frac{1+\delta}{2}\|M-\bar{M}\|_F^2\right]\\
    &\leq -(r-k) \lambda_r^2(M) + (r-k)\lambda_{max}(N_4)^2 + \mathrm{poly}(\lambda)\\
    &\leq -(r-k) \lambda_r^2(M) + (r-k)\lambda_{max}(N)^2 + \mathrm{poly}(\lambda).
\end{align*}
If $k=r$, then the above inequality and inequality \eqref{eqn:7-13} imply that
\[ \mathrm{poly}(\lambda) \geq K = O(\alpha^2), \]
which contradicts the assumption that $\lambda$ is small. Hence, we conclude that $r-k\geq1$. 
Combining with \eqref{eqn:7-13}, we obtain the upper bound
\begin{align}\label{eqn:7-16}
    \nonumber\lambda_r^2(M) &\leq - \frac{2}{1+\delta}\cdot \frac{K}{r-k} + \lambda_{max}(N)^2 + \frac{1}{r-k} \cdot \mathrm{poly}(\lambda)\\
    &= - \frac{2}{1+\delta}\cdot \frac{K}{r} + \lambda_{max}(N)^2 + \mathrm{poly}(\lambda).
\end{align}

\paragraph{Step III.} In the last step, we combine the relations \eqref{eqn:7-15} and \eqref{eqn:7-16}, which leads to
\begin{align*}
    \sigma_r^4(U) \leq - \frac{2}{1+\delta}\cdot \frac{K}{r} + \frac{1}{(1+\delta)^2}G^2 + \mathrm{poly}(\lambda).
\end{align*}
This means that
\begin{align*}
    \sigma_r^4(U) + \frac{2}{1+\delta}\cdot \frac{K}{r} \leq \frac{1}{(1+\delta)^2}G^2 + \mathrm{poly}(\lambda).
\end{align*}
%
Since $K>0$ has lower bounds that are independent of $\lambda$, we can choose $\lambda$ to be small enough such that
\begin{align*}
    \sigma_r^4(U) + \frac{1}{1+\delta}\cdot \frac{K}{r} \leq \frac{1}{(1+\delta)^2}G^2.
\end{align*}
However, considering the assumption \eqref{eqn:cond-sym-2}, we have
\begin{align*}
    \sigma_r^4(U) &\geq \frac{1}{(1+\delta)^2} \left[ G - \frac{(1+\delta)\lambda^2}{4G^2} - \lambda \right]^2 = \frac{1}{(1+\delta)^2} G^2 - 2\lambda \cdot G + \mathrm{poly}(\lambda)\\
    &\geq \frac{1}{(1+\delta)^2} G^2 - 2\lambda \cdot (1+\delta)C^2 + \mathrm{poly}(\lambda) = \frac{1}{(1+\delta)^2} G^2 + \mathrm{poly}(\lambda),
\end{align*}
where the second inequality is due to $G\leq (1+\delta)C^2$, which can be proved similar to the asymmetric case. The above two inequalities cannot hold simultaneously when $\lambda$ is small enough. This contradiction means that the condition \eqref{eqn:cond-sym-1} holds by choosing 
\[ 0<\lambda \leq \lambda_0 (\delta,\sigma_r(M^*_s),\|M^*_s\|_F,\alpha,C), \]
for a small enough positive constant $\lambda_0 (\delta,\sigma_r(M^*_s),\|M^*_s\|_F,\alpha,C)$.

\end{proof}

%
\begin{proof}[Proof of Theorem \ref{thm:strict-saddle-2}]
We first choose 
\[ C:= \left[\frac{2(1+\delta)}{1-\delta} \|U^*(U^*)^T\|_F^{2} \right]^{1/4}. \]
Then, we select $\lambda_1$ as 
\[  \lambda_1(\delta,r,\sigma_r(M^*_s),\|M^*_s\|_F,\alpha) := \min\left\{\lambda_0(\delta,r,\sigma_r(M^*_s),\|M^*_s\|_F,\alpha,C), \frac{(1-\delta)C^3}{2\sqrt{r}} \right\}. \]
Finally, we combine Lemmas \ref{lem:7-1}-\ref{lem:7-2} to get the bounds for the gradient and the Hessian.
\end{proof}

\end{appendix}

\end{document}